\documentclass[numbers,webpdf,arxiv]{ima-authoring-template-arxiv}
\usepackage{amsthm}
\usepackage{amsmath, amsfonts}
\usepackage{booktabs}

\usepackage{algorithm}
\usepackage{algpseudocode}
\usepackage{bm}
\usepackage{graphicx}

\graphicspath{{Fig/}}

\newtheoremstyle{mystyle}
  {6pt}   
  {6pt}   
  {\itshape} 
  {}       
  {\bfseries} 
  {.}      
  { }      
  {}       

\theoremstyle{mystyle}%
\newtheorem{theorem}{Theorem}
\newtheorem{lemma}[theorem]{Lemma}
\newtheorem{corollary}{Corollary}[theorem]

\newtheorem{remark}{Remark}%

\theoremstyle{remark}

\numberwithin{equation}{section}

\begin{document}

\DOI{DOI HERE}
\copyrightyear{2021}
\vol{00}
\pubyear{2021}
\access{Advance Access Publication Date: Day Month Year}
\appnotes{Paper}
\copyrightstatement{Published by Oxford University Press on behalf of the Institute of Mathematics and its Applications. All rights reserved.}
\firstpage{1}


\title[Physics-Informed Neural Networks for SH-Wave Dispersion]{A Hybrid Physics-Informed Neural Network Framework for Computing Dispersion Relations of SH Waves in Generalized Hetrogeneous Layered Media with Applications}

\author{Subhajyoti Sarkar*
\address{\orgdiv{Department of Mathematics and Computing}, \orgname{Indian Institute of Technology (Indian School of Mines) Dhanbad}, \orgaddress{\street{Dhanbad}, \postcode{826004}, \state{Jharkhand}, \country{India}}}}
\author{Santimoy Kundu 
\address{\orgdiv{Department of Mathematics and Computing}, \orgname{Indian Institute of Technology (Indian School of Mines) Dhanbad}, \orgaddress{\street{Dhanbad}, \postcode{826004}, \state{Jharkhand}, \country{India}}}}

\authormark{SUBHAJYOTI SARKAR AND SANTIMOY KUNDU}

\corresp[*]{Corresponding author: \href{sarkarsubhajyoti780@gmail.com}{sarkarsubhajyoti780@gmail.com}}

\received{Date}{0}{Year}
\revised{Date}{0}{Year}
\accepted{Date}{0}{Year}


\abstract{ This work presents a mathematical and computational framework for computing the dispersion relations of shear horizontal (SH) waves in continuously varying heterogeneous layered structures. The approach isolates the contribution of the heterogeneous layer from the complete dispersion relation, learns this contribution using a physics-informed neural network (PINN) and subsequently incorporates the trained model to determine the complete dispersion relation. The mathematical properties of the discretized problem are investigated, including the singularity of the finite-difference system and the oscillatory behavior of the layer solution, while a generalization-error estimate is established for the PINN approximation. The framework is first tested on a seismological configuration consisting of a heterogeneous sandstone layer over a granite half-space, where exponential heterogeneity is considered with independent variation rates in shear modulus and density. The proposed approach is validated against analytical solutions in special cases, while for general configurations the Haskell matrix method demonstrates convergence toward the continuously varying dispersion relation predicted by the PINN as the number of homogeneous sublayers increases. Parametric studies further confirm consistency with the underlying physics. An important feature of the method is that the heterogeneous-layer equation can be trained independently and reused in multiple settings. To demonstrate this, the same trained network is coupled with piezoelectric and piezomagnetic substrates governed by fundamentally different physical laws, highlighting the potential of the PINN framework as a reusable computational module for dispersion analysis in heterogeneous layered media. }



\maketitle

\section{Introduction}
Dispersion relations play a crucial role in understanding the nature of wave propagation in layered structures, as they directly contain information about the frequency-dependent phase velocity of propagating waves. In particular, shear horizontal (SH) waves are of fundamental importance in a wide range of applications. In seismology, SH waves are commonly used for subsurface characterization \cite{yin2020joint,li2019wave}, where dispersion curves obtained from field data are matched with theoretically predicted dispersion curves to infer the subsurface properties responsible for the observed dispersive behavior. SH waves are also widely used in surface acoustic wave (SAW) devices, where their dispersion characteristics govern key device properties such as phase velocity, frequency response, and mode confinement. In such systems, careful engineering of layered piezoelectric and related material structures allows control over dispersion behavior, enabling the design of filters, resonators, and sensors with tailored frequency selectivity and high sensitivity \cite{mandal2022saw}.

Owing to its widespread applicability, accurate modeling of dispersion relations in realistic heterogeneous media is of considerable importance. In practical settings, the material properties of a medium may vary continuously with depth, and the assumption of homogeneous layers may not adequately represent the underlying structure. A variety of depth-dependent heterogeneity models have therefore been considered in the literature. Exponential variations of material properties have been widely employed in the analysis of SH and Love waves \cite{dhua2016wave,singh2025shear,deliktas2023nonlinear}. In particular, Deliktas-Ozdemir considered a layered medium with a vertically heterogeneous half-space and examined exponential variations in the linear and nonlinear material properties and their influence on Love-wave propagation and nonlinear wave evolution \cite{deliktas2023comparison}. Polynomial representations, including linear and quadratic variations, have also been used to describe continuously varying material properties \cite{rajak2022study,gupta2013propagation}. Other functional forms have likewise been considered, including hyperbolic-type profiles such as $\cosh$-based models \cite{dutta1963love}. These studies demonstrate the wide range of functional forms that have been employed to represent depth-dependent heterogeneity and, consequently, the influence of the assumed material profile on the resulting wave characteristics.

For specific choices of the material profiles, the governing variable-coefficient differential equations may admit analytical or semi-analytical treatments. Exponential and polynomial profiles, for example, can lead to formulations that are amenable to analytical reduction or special-function representations. When additional physical effects are incorporated, the resulting mathematical structure can become more involved. For instance, in formulations involving nonlocal effects, Bessel-function representations have been employed in the analysis of Love-wave propagation with polynomial-type heterogeneity \cite{pramanik2024love}. Such analytical formulations are attractive because, once the functional form of the heterogeneity has been prescribed, they can provide an efficient means of evaluating the corresponding dispersion relation. However, the analytical treatment is generally associated with particular functional forms for which the governing equation can be reduced to a suitable canonical or special-function problem. Consequently, changing the prescribed heterogeneity may require a different analytical formulation.

To accommodate more general depth-dependent profiles, numerical approaches have also been considered. Kie\l{}czy'nski et al. (\cite{kielczynski2016propagation}) formulated Love-wave propagation in functionally graded media as a direct Sturm-Liouville problem and solved it using the finite-difference method. Their formulation allows several prescribed profiles to be treated within the same numerical framework. In the finite-difference approach, the continuous differential problem is converted into a generalized matrix eigenvalue problem, whereas in the transfer-matrix approach the heterogeneous medium is represented by a sequence of homogeneous sublayers. Polynomial expansion techniques provide another numerical alternative. Bou Matar et al. (\cite{bou2013legendre}) employed Legendre polynomial expansions in finite layers together with Laguerre polynomial expansions in a semi-infinite substrate, leading to a matrix eigenvalue formulation from which complex wavenumbers can be obtained directly for a prescribed frequency. These developments illustrate that numerical formulations can extend the treatment of heterogeneous wave propagation beyond cases for which closed-form analytical solutions are available.

Several approximate and asymptotic approaches have also been employed for variable-coefficient wave equations. Liu et al. (\cite{liu2005propagation}) used a WKB-type approximation under assumptions appropriate to the corresponding asymptotic regime . Nie et al. (\cite{nie2015effect}) considered a Frobenius-series approach, in which the solution is constructed through a series expansion about a suitable point . Similarly, Shuvalov et al. (\cite{shuvalov2008state}) employed a Peano-series representation for wave propagation in heterogeneous media . These approaches provide useful analytical and asymptotic insight into wave propagation in continuously varying media and extend the range of problems that can be treated beyond a limited set of exactly solvable profiles. At the same time, their solution procedures involve problem-dependent approximations or series representations whose coefficients and evaluation depend on the governing parameters and the assumed material variation.

The choice of computational formulation becomes particularly important when dispersion relations are required repeatedly over a range of frequencies, wavenumbers, or material parameters, as occurs in parameter estimation and inversion. In such applications, the forward dispersion problem must be evaluated for a large number of candidate material models. A methodology that requires a new analytical reduction, asymptotic expansion, discretized eigenvalue problem, or other problem-specific construction for each material configuration may therefore become computationally demanding when many forward evaluations are required. Moreover, the depth dependence of material properties encountered in practical applications need not conform to a prescribed exponential, polynomial, or other canonical form. Such properties may instead be obtained from observations, numerical models, or interpolated measurements and may exhibit non-monotonic or otherwise irregular variations with depth. This motivates the development of a framework that can accommodate a broad class of depth-dependent material profiles while retaining the computational efficiency desirable for repeated dispersion calculations.

This motivates us to explore an alternative approach for obtaining dispersion relations in heterogeneous media in which the shear modulus $\mu$ satisfies
$\mu \in C^{1}([0,H]),$
while the density $\rho$ is assumed to satisfy
$\rho \in C([0,H]),$
where $H$ denotes the thickness of the heterogeneous layer. The objective is to develop a framework that is not restricted to a particular functional form of heterogeneity, but can accommodate a general class of admissible material profiles satisfying these regularity assumptions. In particular, the same framework should be capable of treating different forms of depth dependence without requiring a separate analytical reduction for each prescribed profile.

Recent developments in physics-informed neural networks (PINNs) \cite{raissi2019physics} provide a promising approach in this direction. By incorporating the governing differential equations together with the associated boundary and interface conditions directly into the learning process, PINNs provide a flexible framework for approximating solutions of differential equations with variable coefficients. This makes them a natural candidate for representing the dependence of the wave field, and consequently the dispersion relation, on both the material profiles and the wave-propagation parameters. Our main objective is therefore to investigate the use of PINNs for computing dispersion relations over a general class of heterogeneous media satisfying $\mu\in C^{1}([0,H])$ and $\rho\in C([0,H])$, and to assess their potential as an efficient framework for repeated dispersion calculations across varying material and wave parameters.

The remainder of this paper is organized as follows. Section \ref{sec2} formulates the SH-wave propagation problem in a heterogeneous layer overlying a homogeneous half-space and derives the corresponding dispersion relation. Section \ref{sec3} presents the proposed hybrid physics-informed neural network framework, including the separation of the heterogeneous-layer problem from the substrate contribution, the numerical formulation, the analysis of singularities of the discretized system, and the oscillatory behavior of the layer solution. The PINN formulation, composite loss function, and generalization-error estimate are subsequently developed, followed by the procedure for computing the dispersion relation from the trained network. Section \ref{sec4} presents numerical experiments and applications of the proposed framework. The first application considers SH-wave propagation in a heterogeneous Earth-material configuration, where the effects of the heterogeneity parameters are investigated and the results are validated against analytical solutions and the Haskell matrix method. The second application demonstrates the reusability of the trained heterogeneous-layer model by coupling it with piezoelectric and piezomagnetic substrates governed by different physical laws. Finally, Section \ref{sec5} summarizes the main findings and discusses the implications and possible extensions of the proposed framework.

\section{Problem Formulation} \label{sec2}

We consider the three-dimensional particle displacement vector as
\[
\mathbf{D}_i(x,y,z,t)=\left(u_i(x,y,z,t),\,v_i(x,y,z,t),\,w_i(x,y,z,t)\right),
\]
where \(i=1\) corresponds to the layer and \(i=2\) corresponds to the half-space, respectively. For the remaining manuscript, we shall stick to this subscript notation. The coordinate axes are placed as shown in Fig.~\ref{fig:1}, with \(x\) and \(y\) being horizontal to the surface and positive \(z\) directed vertically downward.

The stress-strain relation is further given as
\[
\boldsymbol{\tau}_i=\lambda_i(z)\left(\nabla\cdot\mathbf{D}_i\right)\mathbf{I}
+2\mu_i(z)\boldsymbol{\varepsilon}_i.
\]

where \(\lambda_i(z)\) and \(\mu_i(z)\) are positive valued depth-dependent Lam\'e parameters, \(\boldsymbol{\tau}_i\) is the stress tensor and the corresponding infinitesimal strain tensor is given by
\[
\boldsymbol{\varepsilon}_i
=
\frac{1}{2}
\left[
\nabla \mathbf{D}_i+
(\nabla \mathbf{D}_i)^T
\right].
\]

Considering Love waves, which are horizontally polarized shear waves, we assume propagation along the $x$-direction, with particle motion polarized along the $y$-direction, as illustrated in Fig.~\ref{fig:1}. Accordingly, only the $y$-component of the displacement field is non-zero and can be expressed as
\[
v_i=v_i(x,z,t),
\]
while all field quantities are independent of the $y$-coordinate, such that
\[
\frac{\partial}{\partial y}=0.
\]
\begin{figure}[htbp]
    \centering
    \begin{minipage}{0.6\textwidth}
        \centering
        \includegraphics[width=\textwidth]{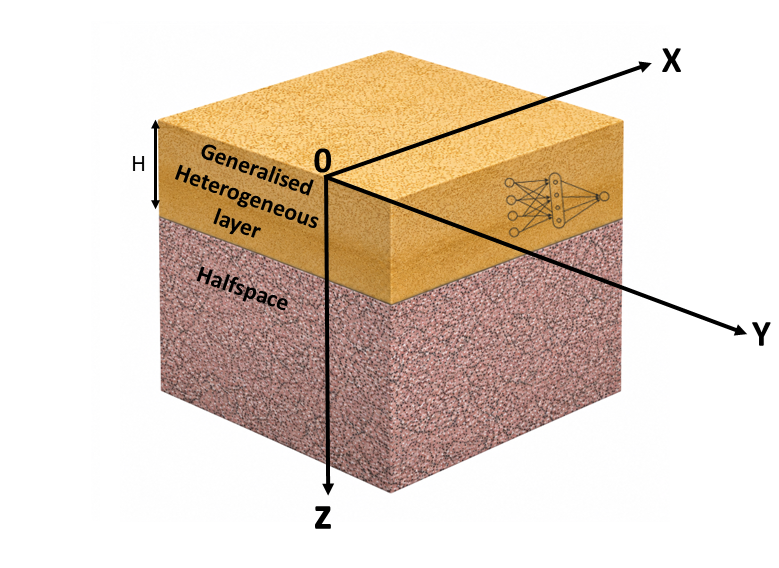}
    \end{minipage}

    \vspace{0.3cm}
    \caption{Basic geometry of the proposed problem.}
    \label{fig:1}
\end{figure}

The corresponding shear stress-strain relations therefore reduce to
\[
\tau_i^{yx}=\mu_i(z)\frac{\partial v_i}{\partial x},
\qquad
\tau_i^{yz}=\mu_i(z)\frac{\partial v_i}{\partial z}.
\]

The general equation of motion is given by
\[
\rho_i(z)\frac{\partial^2\mathbf{D}_i}{\partial t^2}
=
\nabla\cdot\boldsymbol{\tau}_i,
\]
where \(\rho_i(z)>0\) is the depth-dependent density. 

For the special case of Love waves, only the transverse displacement component survives and the governing equation of motion reduces to
\[
\rho_i(z)\frac{\partial^2 v_i}{\partial t^2}
=
\frac{\partial \tau_i^{yx}}{\partial x}
+
\frac{\partial \tau_i^{yz}}{\partial z}.
\]

Using the stress-strain relations, we obtain the equation of motion in terms of the displacement component \(v_i\) as follows:
\[
\rho_i(z)\frac{\partial^2 v_i}{\partial t^2}
=
\frac{\partial}{\partial x}
\left[
\mu_i(z)\frac{\partial v_i}{\partial x}
\right]
+
\frac{\partial}{\partial z}
\left[
\mu_i(z)\frac{\partial v_i}{\partial z}
\right].
\]

Now, we consider the plane wave solution
\[
v_i(x,z,t)=V_i(z)e^{i(kx-\omega t)}.
\]

Therefore, the equation of motion reduces to
\[
-\rho_i(z)\omega^2V_i
=
-\mu_i(z)k^2V_i
+\mu_i'(z)\frac{dV_i}{dz}
+\mu_i(z)\frac{d^2V_i}{dz^2}.
\]

Hence,
\[
\frac{d^2V_i}{dz^2}
+
\frac{\mu_i'(z)}{\mu_i(z)}
\frac{dV_i}{dz}
+
\left(
\frac{\rho_i(z)}{\mu_i(z)}\omega^2-k^2
\right)V_i=0.
\]

For this problem, we consider
\[
\mu_2(z)=\mu_2,
\qquad
\rho_2(z)=\rho_2.
\]

Therefore, the final displacement equations become
\[
\frac{d^2V_1}{dz^2}
+
f(z)
\frac{dV_1}{dz}
+
g(z,k,\omega)V_1=0,
\]

and
\[
\frac{d^2V_2}{dz^2}
-
\alpha_2^2(k,\omega)V_2=0.
\]

where
\[
f(z)=\frac{\mu_1'(z)}{\mu_1(z)},
\qquad
g(z,k,\omega)=
\left(
\frac{\rho_1(z)}{\mu_1(z)}\omega^2-k^2
\right)
\]

and
\[
\alpha_2^2
=
\left(
k^2-\frac{\rho_2\omega^2}{\mu_2}
\right).
\]
The boundary conditions associated with the problem are as follows.

At the free surface \(z=0\), the traction must vanish. Therefore,
\[
\tau_1^{yz}=0
\quad \text{at } z=0.
\]

Using the stress-strain relation, we obtain
\[
\mu_1(0)\frac{dV_1}{dz}=0,
\]

which further reduces to
\begin{equation}
\frac{dV_1}{dz}=0
\quad \text{at } z=0.
\label{stress_free}\end{equation}

Further, the displacement must satisfy the radiation condition in the half-space, i.e.,
\begin{equation}
V_2(z)\to0
\quad \text{as } z\to\infty.
\label{radiation_condition}\end{equation}

Since the wave propagates through both the layer and the underlying half-space, appropriate continuity conditions must be satisfied at their interface, $z=H$. These interface conditions are given by:

\begin{enumerate}
\item Displacement continuity
\begin{equation}
V_1(H)=V_2(H),
\label{disp_continuity}\end{equation}

\item Stress continuity
\[
\tau_1^{yz}=\tau_2^{yz}
\quad \text{at } z=H,
\]

which gives
\begin{equation}
\mu_1(H)\frac{dV_1}{dz}\Bigg|_{z=H}
=
\mu_2\frac{dV_2}{dz}\Bigg|_{z=H}.
\label{stress_continuity}\end{equation}
\end{enumerate}

Thus, the existence of Love waves at a given frequency \(\omega\) implies that there exists a corresponding wavenumber \(k(\omega)\)  such that non-trivial solutions

\[
V_1(z)\in C^2[0,H],
\qquad
V_2(z)\in C^2[H,\infty)
\]

exist satisfying

\begin{equation}
\frac{d^2V_1}{dz^2}
+f(z)\frac{dV_1}{dz}
+g(z,k,\omega)V_1=0,
\qquad z\in[0,H],
\label{eq1}\end{equation}

\begin{equation}
\frac{d^2V_2}{dz^2}
-\alpha_2^2V_2=0,
\qquad z\in[H,\infty),
\label{eq2}\end{equation}

subjected to boundary conditions

\begin{equation}
\frac{dV_1}{dz}=0
\quad \text{at } z=0
\text{ and }
V_2(z)\to0
\quad \text{as } z\to\infty,
\label{bc1}\end{equation}

and interface conditions

\begin{equation}
V_1(H)=V_2(H),
\qquad
\mu_1(H)\frac{dV_1}{dz}
=
\mu_2\frac{dV_2}{dz}
\quad \text{at } z=H.
\label{bc2}\end{equation}
The set of admissible pairs \((\omega,k(\omega))\) for which Love waves exist is known as the dispersion relation. Finding the dispersion relation for generalized layered heterogeneity, i.e.,
$
\mu_1(z)\in C^1[0,H]$ and $\rho_1(z)\in C[0,H]$ is the main objective of this paper. The current formulation is developed for an elastic heterogeneous layer over a homogeneous half-space for mathematical clarity. Nevertheless, the governing framework developed here can naturally be extended to other configurations involving complex constitutive behavior and heterogeneity profiles.

\section{Proposed Framework} \label{sec3}
Our first target is to divide the boundary value problem given by Eq.~(\ref{eq1}) and Eq.~(\ref{eq2}), subjected to the boundary condition Eq.~(\ref{bc1}) and the interface condition Eq.~(\ref{bc2}), into two independent boundary value problems. We achieve this by taking
\[
V_1(H)=V_2(H)=1.
\]

With this normalization, we have separated the original boundary value problem into the following two independent problems.

\subsubsection*{ODE 1:}

\begin{equation*}
\frac{d^2V_1}{dz^2}
+f(z)\frac{dV_1}{dz}
+g(z,k,\omega)V_1=0,
\end{equation*}

subject to
\begin{equation}
\frac{dV_1}{dz}\Bigg|_{z=0}=0,
\qquad
V_1(H)=1.
\label{ode1}\end{equation}

\subsubsection*{ODE 2:}

\[
\frac{d^2V_2}{dz^2}
-\alpha_2^2(k,\omega)V_2=0,
\]

subject to
\begin{equation}
V_2(H)=1,
\qquad
V_2(z)\to0
\quad \text{as } z\to\infty.
\label{ode2}\end{equation}

Let the solution of Eq.~(\ref{ode1}) be denoted by \(V_1(z)\) and the solution of Eq.~(\ref{ode2}) by \(V_2(z)\).

For physically admissible solutions, the stress continuity condition at the interface must be satisfied, i.e.,
\begin{equation}
\mu_1(H)\frac{dV_1}{dz}\Bigg|_{z=H}
=
\mu_2\frac{dV_2}{dz}\Bigg|_{z=H}.\label{dispersion}
\end{equation}

This condition yields the required dispersion relation. We now present a few special cases where analytical solution of $\frac{dV_1}{dz}$ and $\frac{dV_2}{dz}$ are present at $z=H$ to demonstrate the applicability of the proposed framework.

\subsection{Validation Examples}
\subsubsection{Example 1}
Consider a layer with exponential heterogeneity of the same rate in both shear modulus and density, i.e.,
\[
\mu(z)=\mu_0e^{az},
\qquad
\rho(z)=\rho_0e^{az}.
\]

Where $\mu_0$ and $\rho_0$ are the corresponding  values of the shear modulus and density at the surface. Thus, according to Eq.~(\ref{eq1}),
\[
f(z)=a,
\qquad
g(z)=\left(\frac{\rho_0}{\mu_0}\omega^2-k^2\right).
\]

Hence, the first governing equation becomes
\[
\frac{d^2V_1}{dz^2}
+a\frac{dV_1}{dz}
+\left(
\frac{\rho_0}{\mu_0}\omega^2-k^2
\right)V_1=0,
\]

subject to
\[
\frac{dV_1}{dz}\Bigg|_{z=0}=0,
\qquad
V_1(H)=1.
\]
The general solution is therefore,
\[
V_1(z)
=
Ae^{-\frac{az}{2}}
\cos\left(
z\sqrt{d_1^2-\frac{a^2}{4}}\,
\right)
+
Be^{-\frac{az}{2}}
\sin\left(
z\sqrt{d_1^2-\frac{a^2}{4}}\,
\right).
\]

Using the boundary conditions \eqref{ode1}, we obtain the the corresponding particular solution for \(V_1(z)\) as follows:
\begin{equation}
V_1(z)=
\frac{
e^{-a(z-H)/2}
\left[
\cos\!\left(z\sqrt{d_1^2-\frac{a^2}{4}}\right)
+
\frac{a}{2\sqrt{d_1^2-\frac{a^2}{4}}}
\sin\!\left(z\sqrt{d_1^2-\frac{a^2}{4}}\right)
\right]
}
{
\cos\!\left(H\sqrt{d_1^2-\frac{a^2}{4}}\right)
+
\frac{a}{2\sqrt{d_1^2-\frac{a^2}{4}}}
\sin\!\left(H\sqrt{d_1^2-\frac{a^2}{4}}\right)
}.
\label{first}\end{equation}

For the homogeneous isotropic half-space, the governing equation becomes
\[
\frac{d^2V_2}{dz^2}
-\alpha_2^2(k,\omega)V_2=0.
\]

Therefore,
\[
V_2(z)
=
Ce^{\alpha_2 z}
+
De^{-\alpha_2 z}.
\]

 Due to the radiation boundary condition \eqref{ode2} we have the following admissible solution :  
\begin{equation}
V_2(z)
=
e^{\alpha_2 (H-z)}.
\label{second}\end{equation}

Using the stress equality condition \eqref{dispersion} we get the dispersion relation as follows:
\begin{equation}
\mu_1(H)
\frac{
d_1^{\,2}
\sin\!\left(H\sqrt{d_1^2-\frac{a^2}{4}}\right)
}
{
\sqrt{d_1^2-\frac{a^2}{4}}
\cos\!\left(H\sqrt{d_1^2-\frac{a^2}{4}}\right)
+
\frac{a}{2}
\sin\!\left(H\sqrt{d_1^2-\frac{a^2}{4}}\right)
}
=
\mu_2\alpha_2.\label{case1}
\end{equation}
This dispersion relation coincides with the dispersion relation in \cite{sadab2023analytical} under suitable asumptions. 

\subsubsection{Example 2}
Further, setting $a=0$ in \eqref{case1} reduces the dispersion relation to the classical Love wave \cite{stein2003introduction} case of a homogeneous layer over a homogeneous half-space. 

\begin{equation}
\tan\!\left(
H\sqrt{\frac{\rho_0}{\mu_0}\omega^2-k^2}
\right)
=
\frac{\mu_2}{\mu_0}
\frac{
\sqrt{k^2-\frac{\rho_2\omega^2}{\mu_2}}
}{
\sqrt{\frac{\rho_0}{\mu_0}\omega^2-k^2}
}.
\label{Love}\end{equation}

Since the dispersion relation represents a relationship between the spatial frequency (k) and the temporal frequency (\(\omega\)), it is invariant under a multiplicative scaling of the wave amplitudes. Therefore, the normalization
$
V_1(H)=V_2(H)=1
$
may be imposed without any loss of generality. This normalization does not affect the resulting dispersion relation, as only the relative amplitudes and continuity conditions enter the derivation.

From a physical viewpoint, this choice corresponds to scaling the displacement field such that the wave amplitude at the interface (z=H) is unity. Such a normalization is particularly convenient when considering waves generated by a distant source, where the absolute amplitude is not of primary interest and only the phase velocity and dispersion characteristics are sought. Consequently, the interface displacement is taken as the reference amplitude, allowing the analysis to focus solely on the propagation properties of the Love waves. 

We can therefore see that, once a particular solution \(V_1(z)\) of \eqref{ode1} has been obtained, it can be used to evaluate the dispersion relation \eqref{dispersion}. For the homogeneous half-space \eqref{ode2} reduces to

\[
\frac{d^2V_2}{dz^2}(H)-\alpha_2V_2=0,
\]
for decaying solution in the halfspace we further require $\frac{\omega^2}{k^2}<\frac{\mu_2}{\rho_2}$. Further the interfacial traction condition becomes:

\[
\mu_1(H)\frac{dV_1}{dz}(H)
=-\mu_2\alpha_2.
\]

Hence, the dispersion relation can be written as

\begin{equation*}
\mu_1(H)\frac{dV_1}{dz}(H)
+\mu_2\alpha_2=0 . 
\end{equation*}

Therefore, for a prescribed frequency \(\omega\), the admissible wavenumbers (k) are obtained by solving

\begin{equation}
\mu_1(H)\frac{dV_1}{dz}(H)
+\mu_2
\sqrt{k^2-\frac{\rho_2\omega^2}{\mu_2}}
=0 \text{ subjected to } \frac{\omega^2}{k^2}<\frac{\mu_2}{\rho_2} .
\label{disp}\end{equation}

The quantity \(V_1(z)\) is determined from \eqref{ode1} subject to the boundary conditions \(V_1'(0)=0\) and \(V_1(H)=1\). Further for the lower bound of $\frac{\omega^2}{k^2}$ we have the following theorem:   
\newpage
\begin{theorem}\label{thm:theorem1}
    If $(\omega,k)$ satisfies \eqref{disp} then $\frac{\omega^2}{k^2}\geq \inf \limits_{z\in[0,H]} \frac{\mu_1(z)}{\rho_1(z)}$. 
\end{theorem}
\begin{proof}
: Suppose on the contrary $\frac{\omega^2}{k^2}<\inf \limits_{z\in[0,H]} \frac{\mu_1(z)}{\rho_1(z)}$, then $\rho_1(z)\omega^2-\mu_1(z)k^2<0$. 

We first prove that the solution $V_1$ of \eqref{ode1} is positive in $[0,H]$. If it is not true, then there exist an interval in $[0,H]$ where $V_1$ is negative and thus it attains a negative minimum at some $z_0\in[0,H)$.

Now, $z_0\neq 0$, as if it were so then from \eqref{ode1}, we would get that:
\begin{align*}
&V_1''(z_0)=-\frac{(\rho_1(z_0)\omega^2-\mu_1(z_0)k^2)V_1(z_0)}{\mu_1(z_0)}<0,\end{align*}
contradicting the fact that $z_0$ is a minima. Therefore $z_0\in(0,H)$. Now at this interior minima $z_0$, $V_1''(z_0)>0$ and $V_1'(z_0)=0$. Therefore \[(\mu_1V_1')'(z_0)=\mu_1(z_0)V_1''(z_0)>0,\]
but from \eqref{ode1}, we get 
\[(\mu_1V_1')'(z_0)=-(\rho_1(z_0)\omega^2-\mu_1(z)k^2)V_1(z_0)<0\text{ giving a contradiction.}\]

Therefore \begin{equation}  V_1(z)\geq0 \quad \forall z\in[0,H].\end{equation}
Now on we look at the sign of $V_1'(H)$, as it is the unknown component of \eqref{disp}. 

Applying integration by parts from 0 to $H$ on \eqref{ode1}, we get:
\[\mu_1(H)\frac{dV_1}{dz}(H)=-\int_0^H(\rho_1(z)\omega^2-\mu_1(z)k^2)V_1(z)dz>0.\]
Therefore, \[\mu_1(H)\frac{dV_1}{dz}(H)
+\mu_2
\sqrt{k^2-\frac{\rho_2\omega^2}{\mu_2}}>0,\]
which is a contradiction, as the dispersion relation \eqref{disp} will never be satisfied for such a pair of $(\omega,k)$. This completes the proof.
 \end{proof}
For the existance and uniqueness of \eqref{ode1}, we have the following:
\newpage
\begin{theorem}\label{thm:exist_unique}
    For fixed $\omega>0$, the classical solution of the boundary value problem \eqref{ode1}, exists uniquely $\forall k$ satisfying $\inf \limits_{z\in[0,H]} \frac{\mu_1(z)}{\rho_1(z)}\leq\frac{\omega^2}{k^2}<\frac{\mu_2}{\rho_2}$, except for possibly finitely many values of $k$.\end{theorem}
\begin{proof}
: For a fixed $(\omega,k)$, let $\tilde{V}_1$ be the solution of of the associated intial value problem: 
\begin{equation}
\frac{d^2\tilde{V}_1}{dz^2}
+\frac{\mu_1'(z)}{\mu_1(z)}\frac{d\tilde{V}_1}{dz}
+\left(
\frac{\rho_1(z)}{\mu_1(z)}\omega^2-k^2
\right)\tilde{V}_1=0;\quad 
\tilde{V}_1(0)=1 \text{ }\& \text{ } \tilde{V}_1'(0)=0.
\label{reference}\end{equation}
 Since $\mu_1\in C^1[0,H]$, $\mu_1(z)>0$ and $\rho_1\in C[0,H]$, the standard existence and uniqueness theorem for linear ordinary differential equations guarantees that $\tilde{V}_1$ exists uniquely on $[0,H]$.

We first show that the boundary condition $V_1(H)=1$ can be satisfied if and only if $\tilde{V}_1(H)\neq0$. Indeed, any solution of \eqref{ode1} satisfying $V_1'(0)=0$ is of the form
$
V_1(z)=C \tilde{V}_1(z),
$
for some constant $C$. Therefore, $V_1(H)=1$ requires
$C \tilde{V}_1(H)=1$. Hence, if $\tilde{V}_1(H)\neq0$, the unique solution to \eqref{ode1} is
$
V_1(z)=\frac{\tilde{V}_1(z)}{\tilde{V}_1(H)}.
$

On the other hand, if $\tilde{V}_1(H)=0$, then $V_1(H)=0$ for every solution satisfying $V_1'(0)=0$ and hence the condition $V_1(H)=1$ cannot be satisfied.

Thus we have to look for nontrivial solution $\tilde{V}_1$ satisfying the following: 
\begin{equation}
\frac{d}{dz}\bigg({\mu_1(z)}\frac{d\tilde{V}_1}{dz}\bigg)
+
{\rho_1(z)}\omega^2={\mu_1(z)}k^2
\tilde{V}_1;\quad 
\tilde{V}_1'(0)=0 \text{ }\& \text{ } \tilde{V}_1(H)=0.
\label{sform}\end{equation}
Thus, with eigenvalues  $k^2$, this is a regular Sturm-Liouville problem on the finite interval $[0,H]$. By the  Sturm-Liouville theory, its eigenvalues are discrete and have no finite accumulation point.

Hence, except possibly for finitely many values of $k$, we have $\tilde{V}_1(H)\neq0$ and therefore \eqref{ode1} admits a unique classical solution.
\end{proof}

Let $\mathscr{A}\subset \mathbb{R}$, be set of eigenvalues($k$) satisfying \eqref{sform}, then using the above Theorems~\ref{thm:theorem1} and \ref{thm:exist_unique}, the constraints on the dispersion relation are further updated as follows:
 \begin{equation}
\mu_1(H)\frac{dV_1}{dz}(H)
+\mu_2
\sqrt{k^2-\frac{\rho_2\omega^2}{\mu_2}}
=0; \quad  k\in
\left(
\omega\sqrt{\frac{\rho_2}{\mu_2}},
\frac{\omega}{
\sqrt{\displaystyle\inf_{z\in[0,H]}
\frac{\mu_1(z)}{\rho_1(z)}}}
\right]
\cap
\left(\mathbb{R}_+\setminus\mathscr{A}\right).
\label{disp_main}\end{equation}
\subsection{Numerical Formulation}\label{fdscheme}
For general $\mu_1(z)\in C^1([0,H];(0,\infty))$ and $\rho_1(z) \in C([0,H];(0,\infty))$ analytical solutions for \eqref{ode1} is not always possible, therefore we proceed numerically. The interval $0 \le z \le H$ is discretized using the uniform grid $ z_i = ih,\quad i=0,1,\ldots,N,$ where $h = H/N$ is the uniform grid spacing. This discretization consists of $N+1$ grid points. Applying the second-order central difference approximation to \eqref{ode1} at the interior grid points
$i=1,\ldots,N-1$ yields

\[
C_iV_{i-1}+B_iV_i+A_iV_{i+1}=0,
\]

where

\begin{equation}
C_i=\frac{1}{h^2}-\frac{f(z_i)}{2h},\qquad
B_i=g(z_i,k,\omega)-\frac{2}{h^2},\qquad
A_i=\frac{1}{h^2}+\frac{f(z_i)}{2h}.
\label{coeff_matrix}
\end{equation}

The Neumann boundary condition at $z=0$ gives

\[
\left[g(0,k,\omega)-\frac{2}{h^2}\right]V_0
+\frac{2}{h^2}V_1=0.
\]

Accordingly, the coefficients corresponding to the first row of the linear system are defined as

\[
B_0=g(0,k,\omega)-\frac{2}{h^2},\qquad
A_0=\frac{2}{h^2},\qquad
C_0=0.
\]

The Dirichlet boundary condition at $z=H$ is simply

\[
V_N=1.
\]

The resulting linear system is given as follows:

\begin{equation}
\begin{bmatrix}
B_0 & A_0 & 0 & \cdots & 0 & 0\\
C_1 & B_1 & A_1 & \ddots & \vdots & \vdots\\
0 & C_2 & B_2 & \ddots & 0 & 0\\
\vdots & \ddots & \ddots & \ddots & A_{N-2} & 0\\
0 & \cdots & 0 & C_{N-1} & B_{N-1} & A_{N-1}\\
0 & \cdots & 0 & 0 & 0 & 1
\end{bmatrix}
\begin{bmatrix}
V_0\\
V_1\\
V_2\\
\vdots\\
V_{N-1}\\
V_N
\end{bmatrix}
=
\begin{bmatrix}
0\\
0\\
0\\
\vdots\\
0\\
1
\end{bmatrix},
\label{matrix}\end{equation}

which is solved using the Thomas algorithm. 
 For a fixed $\omega$, let $M(k)$ denote the coefficient matrix in \eqref{matrix}. We now present results regarding its singularities.

\begin{theorem}\label{thm:theorem2}
 The set
$
\mathcal S=\{k\in\mathbb R:\det M(k)=0\}
$
has Lebesgue measure zero.
\end{theorem}

\begin{proof}
: Let $D_i(k)$ denote the determinant of the $i^\text{th}$ leading principal submatrix of $M(k)$. Then
\[
D_0=B_0,\qquad
D_1=B_0B_1-A_0C_1,
\]
and, for $i\ge2$, the determinants satisfy the recurrence
\[
D_i=B_iD_{i-1}-A_{i-1}C_iD_{i-2}.
\]
Since $A_i$, $B_i$ and $C_i$ are polynomials in $k$, it follows by induction that each $D_i$ is also a polynomial . In particular, $D_{N-1}$ is a polynomial whose highest-degree term in $k$ is $(-1)^N$.
 Hence $D_{N-1}\not\equiv0$.

Since
$
\det M(k)=D_{N-1}(k),
$
it follows that $\det M(k)$ is a nonzero polynomial. Therefore, its zero set
$
\mathcal S=\{k\in\mathbb R:\det M(k)=0\}
$
has Lebesgue measure zero.
\end{proof}
Since the singular set has Lebesgue measure zero, a singular parameter $k$ can almost always be avoided by an arbitrarily small perturbation. The following corollary follows immediately.
\newpage
 \begin{corollary}\label{corollary}
Let $k\in \mathbb{R}$ be a random variable whose distribution is absolutely continuous. Then the probability of attaining singularity is zero i.e, 
$
\mathbb{P}\!\left(\det M(k)=0\right)=0.
$
Equivalently, the coefficient matrix $M(k)$ is nonsingular almost surely.
\end{corollary}
 \begin{proof}
: By Theorem~\ref{thm:theorem2}, the singular set
$
\mathcal S=\{k\in\mathbb R:\det M(k)=0\}
$
has Lebesgue measure zero. Since the distribution of $k$ is absolutely continuous with respect to the Lebesgue measure,
$
\mathbb P\big(k\in\mathcal S\big)=0.
$
Equivalently,
$
\mathbb P\!\left(\det M(k)=0\right)=0,
$
and hence $M(k)$ is nonsingular almost surely.
\end{proof}

\begin{remark}Although the Thomas algorithm has a linear computational time complexity of $\mathscr{O}(N)$, the overall computational cost can still become significant. This is because the finite dispersion relation \eqref{disp_main} is implicit and its solution requires a root-finding procedure in which the tridiagonal linear system must be solved repeatedly. Furthermore, in practical applications such as the inversion of observed dispersion data, the dispersion curves must be computed and matched numerous times during the optimization process. Our next target is thus to improve this computational complexity by employing Physics Induced Neural Networks (PINNs).
\end{remark}
Before moving forward, we present a result regarding the oscillatory nature of the ODE1 \eqref{ode1}.
\begin{theorem}\label{thm:theorem3}
If
$
\frac{\omega^2}{k^2}=c^2>\inf\limits_{z\in[0,H]}\frac{\mu_1(z)}{\rho_1(z)},
$
then the solution of \eqref{ode1} becomes increasingly oscillatory as $\omega$ or $H$ increases.
\end{theorem}
\begin{proof}
: Equation \eqref{ode1} can be written in the self-adjoint form
\[
\frac{d}{dz}\left(\mu_1(z)\frac{dV_1}{dz}\right)
+\left(\rho_1(z)\omega^2-\mu_1(z)k^2\right)V_1=0,
\qquad z\in[0,H].
\]

Introducing the normalized coordinate
$\xi=\frac{z}{H},$

so that $\xi\in[0,1]$. Thus the above equation becomes
\[
\frac{d}{d\xi}\left(\mu_1(\xi)\frac{dV_1}{d\xi}\right)
+H^2\left(\rho_1(\xi)\omega^2-\mu_1(\xi)k^2\right)V_1=0.
\]
Since $c=\omega/k$, this may be rewritten as
\[
\frac{d}{d\xi}\left(\mu_1(\xi)\frac{dV_1}{d\xi}\right)
+H^2\omega^2\mu_1(\xi)
\left(
\frac{\rho_1(\xi)}{\mu_1(\xi)}
-\frac1{c^2}
\right)V_1=0.
\]

By hypothesis,
$
\mu_1(\xi)
\left(
\frac{\rho_1(\xi)}{\mu_1(\xi)}
-\frac{1}{c^2}
\right)>0
\quad\text{on }[0,1].
$
Now let $\omega>\omega_0>0$. The coefficient of $V_1$ corresponding to $\omega$ is pointwise larger than that corresponding to $\omega_0$. Hence, by the Sturm Comparison Theorem, every nontrivial solution associated with $\omega$ possesses at least as many zeros on $[0,1]$ as the corresponding solution associated with $\omega_0$. Thus, increasing $\omega$ (while $k$ fixed) increases the oscillatory behaviour of the solution.

The same conclusion follows when $\omega$ is fixed and $H$ increases, since $H$ appears as the positive multiplicative factor $H^2$ in the coefficient of $V_1$.
\end{proof}
The assumption for the above equation is natural, since we are only interested in propagating waves \eqref{disp_main}. We are now in a position to employ a physics-informed neural network (PINN) to solve \eqref{ode1}. Since the trained network can be evaluated rapidly, it provides a fast approximation of the solution $V_1^{'}(H)$, which is subsequently used to compute the dispersion relation \eqref{disp_main}.
\subsection {Hybrid Physics Induced Neural Network Formulation }
Let $\Gamma$ be the $n$ dimentional admissible parameter space associated with the heterogeneity functions $\mu_1(z)$ and $\rho_1(z)$. For every $(\gamma,k)\in\Gamma\times[k_{\min},k_{\max}]$, let
$
V_1^{\gamma,k}(z)
$
denote the corresponding solution of the boundary value problem \eqref{ode1}. Our objective is to construct a neural network approximation of this solution, denoted by
$
V_{1,\theta}^{\gamma,k}(z),
$
where $\theta$ represents the trainable parameters of the neural network. Throughout this work, the layer thickness $H$ is kept fixed. Moreover, for each prescribed angular frequency $\omega$, a separate neural network is trained. This choice is motivated by Theorem~\ref{thm:theorem3}, which shows that increasing $\omega$ and $H$ increases the oscillatory nature of the solution, thereby increasing the complexity of the approximation problem. Consequently, the dependence of the solution on these quantities is suppressed in the notation.

The approximation $V_{1,\theta}^{\gamma,k}$ is realized by a fully connected feed-forward neural network represented as the composition
\[
V_{1,\theta}^{\gamma,k}
=
T_L\circ\sigma\circ T_{L-1}\circ\cdots\circ\sigma\circ T_1,
\]
where $T_i(z)=W_i z+b_i$ are affine transformations parameterized by the trainable weights ($W_i$) and biases ($b_i$) contained in $\theta$ and $\sigma:\mathbb{R}\rightarrow\mathbb{R}$ is a sufficiently smooth activation function. 
\subsection{Loss Function}

The proposed loss function is constructed from five independent losses corresponding to the residue of the governing differential equation and prescribed boundary conditions, the finite difference solution and the end derivative. 

\subsubsection{ODE and boundary residuals }

Since $V_{1,\theta}^{\gamma,k}\in C^2[0,H]$, we define the ODE residue ($\mathcal{R}_{ode}$) and the boundary residues ($\mathcal{R}_{b1}$ and $\mathcal{R}_{b2}$) as follows:
\begin{align}
    &\mathcal{R}_{ode}(z)=\frac{d^2V_{1,\theta}^{\gamma,k}}{dz^2}
+
\frac{\mu_1'(z)}{\mu_1(z)}
\frac{dV_{1,\theta}^{\gamma,k}}{dz}
+
\left(
\frac{\rho_1(z)}{\mu_1(z)}\omega^2-k^2
\right)V_{1,\theta}^{\gamma,k},
\\&\mathcal{R}_{b1}=\frac{dV_{1,\theta}^{\gamma,k}}{dz}\Bigg|_{z=0},
\\&\mathcal{R}_{b2}=V_{1,\theta}^{\gamma,k}(H)-1.
\label{residuals1}\end{align} 
 
\subsubsection{Numerical Residuals}
In addition to the ODE residuals, we incorporate supervision from the finite difference solution obtained in Section~\ref{fdscheme} to increase the training accuracy of the neural network. This is motivated by our ultimate goal of using the trained network to compute the dispersion relation \eqref{disp_main}, where higher prediction accuracy of $\frac{dV_{1,\theta}^{\gamma,k}}{dz}$ is beneficial. Let $V_{1,h}^{\gamma,k}(z)$ denote the numerical solution computed using the numerical scheme \eqref{fdscheme}. This solution serves as a reference to guide the neural network during training and improve the convergence of the optimization process. We measure the discrepancy between the neural network approximation and the finite difference solution in the $H^1$-norm. Furthermore, as the dispersion relation \eqref{disp_main} depends explicitly on the interface derivative $V_1'(H)$, an additional loss component is introduced to directly penalize the error in the predicted derivative at the interface $z=H$. Therefore the finite difference residue is given by:
\begin{equation}
\mathscr{R}_{{FD}}(z)^2=\bigg(V_{1,\theta}^{\gamma,k}(z)-V_{1,h}^{\gamma,k}(z)\bigg)^2+\bigg(\frac{dV_{1,\theta}^{\gamma,k}(z)}{dz}-D_h(V_{1,h}^{\gamma,k}(z))\bigg)^2.
\end{equation}
and we also have the interface residue as follows:
\begin{equation}
    \mathscr{R}_{int}=\frac{dV_{1,\theta}^{\gamma,k}(H)}{dz}-D_h(V_{1,h}^{\gamma,k}(H)).
\end{equation}
Where, $D_h(V_{1,h}^{\gamma,k}(z))$ is finite difference derivative with mesh size $h$ and order of accuracy $O(h^\alpha)$. 

The composite loss function $\mathscr{L}_{\theta}$ is thus given by:
\begin{equation}
\mathscr{L}_{\theta}
=
l_{ode}\underbrace{\|\mathscr{R}_{ode}(z)\|^2_{L^2([0,H])}}_{T_{ode}}
+
l_{b1}\underbrace{\mathscr{R}_{b1}^2}_{T_{b1}}
+
l_{b2}\underbrace{\mathscr{R}_{b2}^2}_{T_{b2}}
+
l_{FD}\underbrace{\|\mathscr{R}_{FD}(z)\|_{L^2([0,H])}^2}_{T_{FD}}
+
l_{int}\underbrace{\mathscr{R}_{int}^2}_{T_{int}}.
\label{loss}\end{equation}
Here $l_{ode}, l_{b1}, l_{b2}, l_{FD}\text{ and }l_{int}$ are the learning-rate coefficients. The quantities $T_{ode}, T_{b1}, T_{b2}, T_{FD}\text{ and }T_{int}$ denote the corresponding training errors. For computing $T_{ode}$ and $T_{FD}$, the corresponding $L^2$ norms are approximated using a quadrature rule. Let $\{z_i\}_{i=1}^{N}\subset[0,H]$ denote the quadrature points and let $\{d_i\}_{i=1}^{N}$ denote the corresponding quadrature coefficients. Then we approximate $T_{ode}$ and $T_{FD}$,
\begin{equation}
\bigg|T_{ode}
-
\sum_{i=1}^{N}d_i\mathscr{R}_{ode}^{\,2}(z_i)\bigg|=\bigg|\int_{0}^{H}\mathscr{R}_{ode}^{\,2}(z)\,dz
-
\sum_{i=1}^{N}d_i\mathscr{R}_{ode}^{\,2}(z_i)\bigg|\leq C_{q_1} N^{-\psi}
\label{quad1}\end{equation}
and similarly,
\begin{equation}
\Bigg|T_{FD}-\int_{0}^{H}\mathscr{R}_{FD}^{\,2}(z)\,dz\Bigg|
=\Bigg|\int_{0}^{H}\mathscr{R}_{FD}^{\,2}(z)\,dz
-
\sum_{i=1}^{N}d_i\mathscr{R}_{FD}^{\,2}(z_i)\Bigg|\leq C_{q_2}N^{-\psi}.
\label{quad2}\end{equation}
Where $C_{q_1}$ and $C_{q_2}$ depends on the regularity of $\mathscr{R}_{ode}$ and $\mathscr{R}_{FD}$ respectively and $\psi>0$ depends on the concerned quadrature rule. 

 For computing $\mathscr{R}_{FD}$ and $\mathscr{R}_{int}$ we need stable data set for the numerical solution ($V_{1,h}^{\gamma,k}$). In view of Theorem~\ref{thm:theorem2}, the set of singular wavenumbers corresponding to any fixed frequency $\omega$ has Lebesgue measure zero. Hence, whenever a singular matrix is detected during the computation, the wavenumber is replaced by a perturbed value
$
\tilde{k}\sim U(k-\epsilon,k+\epsilon),
$
where $\epsilon>0$ is arbitrarily small. Since the probability of sampling a point from a measure-zero subset is zero, the perturbed parameter is nonsingular almost surely. The complete pseudocode of the PINN Algorithim is presented in Algorithm~\ref{alg1}, where Steps~5-10 outline the strategy for handling singularities.       
\begin{algorithm}[H]
\caption{Generalized PINN Algorithm}
\label{alg1}
\begin{algorithmic}[1] 
\Require $\omega,H,\Gamma,k_{\min},k_{\max},\epsilon$ and $N$.
\Ensure Neural Network solution $\theta^*$ of ODE1 \eqref{ode1}.
\State Initialise weights and biases $\theta$.
\State Initialise $s\leftarrow$ TRUE. \Comment{singularity detector}
\While{$e<e_{\max}$} \Comment{$e_{\max}$($\max$ epochs).}
    \State Sample $(\gamma,k)\sim U[\Gamma\times[k_{\min},k_{\max}]]$
    \While{(s$==$TRUE)}
        \If {Computation of $V_{1,h}^{\gamma,k}$ encounters a divide by zero error}
            \State Sample $k\sim U[k-\epsilon,k+\epsilon]$. \Comment{singularity control step}
        \Else
            \State $s\leftarrow$ FALSE.
        \EndIf
    \EndWhile
    \State Evaluate the neural network solution at the collocation points ($z_i$ s.t. $0\leq i\leq N$).
    \State Compute ODE and boundary residuals: $\mathscr{R}_{ode},\mathscr{R}_{b1}$ and $\mathscr{R}_{b2}.$
    \State Compute Numerical Residuals: $\mathscr{R}_{FD}$ and $\mathscr{R}_{int}$.
    \State Compute Loss $\mathscr{L}_{\theta}$ and update $\theta$ by minimizing $\mathscr{L}_{\theta}$.
    \State $e\leftarrow e+1$.
\EndWhile
\end{algorithmic}
\end{algorithm}

\subsubsection{Generalization Error Estimates}
 For $k\notin \mathscr{A}$, the generalization error $\mathscr{E}$ is defined as the $L^2$-norm of the difference between the classical solution $(V_1^{\gamma,k})$ and the corresponding neural network approximation $(V_{1,\theta}^{\gamma,k})$ \cite{mishra2023estimates} i.e,
 \begin{align}
 \mathscr{E}^{\gamma,k}=& \bigg(\int_0^H\bigg|V_1^{\gamma,k}(z)-V_{1,\theta}^{\gamma,k}(z)\bigg|^2dz\bigg)^{\frac{1}{2}} =\bigg(\int_0^H\bigg|\hat{V}_1^{\gamma,k}(z)\bigg|^2dz\bigg)^{\frac{1}{2}}.\label{E_def}\end{align}
 Where $\hat{V}_1^{\gamma,k}=V_1^{\gamma,k}(z)-V_{1,\theta}^{\gamma,k}(z)$ is the error term. The ODE and the boundary residuals \eqref{residuals1} in terms of $\hat{V}_{1}^{\gamma,k}$ are as follows:

 \begin{align}
    &\mathcal{R}_{ode}(z)=\frac{d^2\hat{V}_{1}^{\gamma,k}}{dz^2}
+
\frac{\mu_1'(z)}{\mu_1(z)}
\frac{d\hat{V}_{1}^{\gamma,k}}{dz}
+
\left(
\frac{\rho_1(z)}{\mu_1(z)}\omega^2-k^2
\right)\hat{V}_{1}^{\gamma,k},
\label{ode3}\\&\mathcal{R}_{b1}=\frac{d\hat{V}_{1}^{\gamma,k}}{dz}\Bigg|_{z=0}\label{ode4},
\\&\mathcal{R}_{b2}=\hat{V}_{1}^{\gamma,k}(H).\label{ode5}
\end{align} 
\vspace{-1pt}
We now present bounds w.r.t. the numerical residuals which we will need later for the composite Generalization error estimate.
\newpage
\begin{lemma}\label{lemma5}
For $k\notin\mathscr{A}$, the following estimates hold:
\newline (i) $ \bigg|\frac{d\hat{V}_{1}^{\gamma,k}}{dz}(H)\bigg|\leq O\bigg(\frac{H}{N}\bigg)+\bigg|\mathscr{R}_{int}\bigg|.$\\ (ii) $\|\hat{V}_1^{\gamma,k}\|_{H^1([0,H])}^2\leq O\bigg(\bigg(\frac{H}{N}\bigg)^{\min(\alpha,1)}\bigg)+2\int_0^H\mathscr{R}_{FD}^2dz.$ 
\\(iii) $ |\hat{V}_1^{\gamma,k}(0)|\leq\sqrt{\max\bigg\{\frac{2}{H},2H\bigg\}\bigg( O\bigg(\bigg(\frac{H}{N}\bigg)^{\min(\alpha,1)}\bigg)+2\int_0^H\mathscr{R}_{FD}^2dz\bigg).}$ 
\end{lemma}\vspace{-24pt}\begin{proof}
: We first prove \textnormal{(i)}, 
\begin{align*}
    \bigg|\frac{d\hat{V}_1^{\gamma,k}(H)}{dz} \bigg|&=\bigg|\frac{dV_1^{\gamma,k}(H)}{dz} -\frac{dV_{1,\theta }^{\gamma,k}(H)}{dz}\bigg|
    \\&=\bigg|\frac{dV_1^{\gamma,k}(H)}{dz}-D_h(V_{1,h}^{\gamma,k}(H))+D_h(V_{1,h}^{\gamma,k}(H)) -\frac{dV_{1,\theta }^{\gamma,k}(H)}{dz}\bigg|
    \\&\leq O(h)+|\mathscr{R}_{int}|.
\end{align*}
where the last inequality follows from the triangle inequality and the fact that $\bigg|V_1^{\gamma,k}-V_{1,h}^{\gamma,k}\bigg|\sim O(h^2)$.
We next prove (ii), using Young's inequality, we get:
\begin{align}\nonumber\|\hat{V}_1^{\gamma,k}\|_{H^1([0,H])}^2&=\int_0^H(\hat{V}_1^{\gamma,k})^2+\bigg(\frac{d\hat{V}_1^{\gamma,k}}{dz}\bigg)^2~dz
\nonumber\\&= \int_0^H({V}_1^{\gamma,k}-{V}_{1,h}^{\gamma,k}+{V}_{1,h}^{\gamma,k}-{V}_{1,\theta}^{\gamma,k})^2+\bigg(\frac{d{V}_1^{\gamma,k}}{dz}-D_h(V_{1,h}^{\gamma,k}) +D_h(V_{1,h}^{\gamma,k})-\frac{d{V}_{1,\theta}^{\gamma,k}}{dz} \bigg)^2~dz
\nonumber\\&\leq 2 \int_0^H({V}_1^{\gamma,k}-{V}_{1,h}^{\gamma,k})^2~dz+2\int_0^H({V}_{1,h}^{\gamma,k}-{V}_{1,\theta}^{\gamma,k})^2~dz +2\int_0^H \bigg(\frac{d{V}_1^{\gamma,k}}{dz}-D_h(V_{1,h}^{\gamma,k})\bigg)^2dz \nonumber\\&\quad+2\int_0^H\bigg(D_h(V_{1,h}^{\gamma,k})-\frac{d{V}_{1,\theta}^{\gamma,k}}{dz} \bigg)^2dz
\nonumber\\&\leq  O(h^2) + 2\int_0^H({V}_{1,h}^{\gamma,k}-{V}_{1,\theta}^{\gamma,k})^2~dz  +2\int_0^H \bigg(\frac{d{V}_1^{\gamma,k}}{dz}-D_h(V_{1,h}^{\gamma,k})\bigg)^2dz
\nonumber\\&\quad+2\int_0^H\bigg(D_h(V_{1,h}^{\gamma,k})-\frac{d{V}_{1,\theta}^{\gamma,k}}{dz} \bigg)^2dz\label{eqa}\end{align}
Now, considering $h=\frac{H}{N}<1$, applying Young's inequality, we get:
\begin{align}
    \nonumber\int_0^H \bigg(\frac{dV_1^{\gamma,k}}{dz}-D_hV_{1,h}^{\gamma,k}\bigg)^2~dz&=\int_0^H \bigg(\frac{dV_1^{\gamma,k}}{dz}-D_hV_1^{\gamma,k}+D_hV_1^{\gamma,k}-D_hV_{1,h}^{\gamma,k}\bigg)^2~dz
    \nonumber\\&\leq 2\int_0^H \bigg(\frac{dV_1^{\gamma,k}}{dz}-D_hV_1^{\gamma,k}\bigg)^2~dz+2\int_0^H \bigg(D_hV_1^{\gamma,k}-D_hV_{1,h}^{\gamma,k}\bigg)^2~dz
    \nonumber\\&\leq 2O(h^\alpha)+2O(h)
    \nonumber\\&\leq O(h^{\max(\alpha,1)})
\label{eqb}\end{align}

Using \eqref{eqa} in \eqref{eqb}, we get the required inequality.

Finally, we prove (iii), using Fundamental Theorem of Calculus:
\[\hat{V}_1^{\lambda,k}(0)=\hat{V}_1^{\lambda,k}(z)-\int_0^z\frac{d\hat{V}_1^{\lambda,k}}{ds} ds.\]
Using Cauchy-Squartz inequality we further have the following:
\begin{align}
    \nonumber|\hat{V}_1^{\gamma,k}(0)|&\leq |\hat{V}_1^{\gamma,k}(z)|+\sqrt{z}\bigg\|\frac{d\hat{V}_1^{\gamma,k}}{dz}\bigg\|_{L^2([0,z])}
    \nonumber\\&\leq|\hat{V}_1^{\gamma,k}(z)|+\sqrt{H}\bigg\|\frac{d\hat{V}_1^{\gamma,k}}{dz}\bigg\|_{L^2([0,H])}
\end{align}

Further using Young's inequality and integrating from $0\text{ to }H$ we get:
\begin{align}
    \nonumber\int_0^H |\hat{V}_1^{\gamma,k}(0)|^2 dz \leq 2 \int_0^H|\hat{V}_1^{\gamma,k}(z)|^2dz+2H^2\bigg\|\frac{d\hat{V}_{1}^{\gamma,k}}{dz}\bigg\|^2_{L^2([0,H])}
\end{align}
\begin{align}
    \implies\nonumber |\hat{V}_1^{\gamma,k}(0)|^2 dz &\leq \frac{2}{H} \|\hat{V}_1^{\gamma,k}\|_{L^2([0,H])}^2+2H\bigg\|\frac{d\hat{V}_{1}^{\gamma,k}}{dz}\bigg\|^2_{L^2([0,H])}
    \\&\leq \max \bigg\{\frac{2}{H},2H \bigg\} \|\hat{V}_1^{\gamma,k}\|_{H^1([0,H])}^2
\end{align}
Using (ii) we get the desired inequality.
\end{proof}
Now we provide an estimate of the generalization error based on the per epoch information obtained from Steps 13 and 14 of Algorithm \ref{alg1}.
\begin{theorem}
For $k\notin\mathscr{A}$, suppose that
\[
I(\omega,k)
:=
\inf_{z\in[0,H]}
\left\{
\frac{\rho_1(z)\omega^2}{\mu_1(z)}
-k^2
-\frac{d}{dz}
\left(
\frac{\mu_1'(z)}{\mu_1(z)}
\right)
-\frac{1}{2}
\right\}
>0.
\]
Then the generalization error satisfies
\begin{align}
\begin{aligned}
\left(\mathscr{E}^{\gamma,k}\right)^2
\leq\;&
\frac{1}{I(\omega,k)}
\Bigg\{
|T_{b1}^{1/2}|\bigg(\max\left\{\frac{2}{H},2H\right\}O\left(
\left(\frac{H}{N}\right)^{\min(\alpha,1)}
\right)
+2T_{FD}
+2C_{q_2}N^{-\psi}\bigg)^\frac{1}{2}
\\&+
|T_{b2}^{1/2}|
\left(
O\left(\frac{H}{N}\right)
+
|T_{\mathrm{int}}^{1/2}|
\right)
+
O\left(
\left(\frac{H}{N}\right)^{\min(\alpha,1)}
\right)
+2T_{FD}
+2C_{q_2}N^{-\psi}
\\&
+
\left|
\frac{u_1'(0)}{u_1(0)}
\right|\bigg(\max\left\{\frac{2}{H},2H\right\}
O\left(
\left(\frac{H}{N}\right)^{\min(\alpha,1)}
\right)+2T_{FD}
+2C_{q_2}N^{-\psi}\bigg)
\\&+
\left|
\frac{\mu_1'(H)}{\mu_1(H)}
\right|
T_{b2}
+
\frac{1}{2}
\left(
T_{\mathrm{ode}}
+
C_{q_1}N^{-\psi}
\right)
\Bigg\}.
\end{aligned}
\label{eq:generalization_error_bound}
\end{align}
\end{theorem}
\begin{proof}
: Multiplying \eqref{ode3} by $\hat{V}_1^{\gamma,k}$ and integrating from $0$ to $H$, we get:
\begin{align*}&\bigg(\frac{d\hat{V}_1^{\gamma,k}}{dz}\bigg|_{z=H}\bigg)\hat{V}_1^{\gamma,k}(H)-\bigg(\frac{d\hat{V}_1^{\gamma,k}}{dz}\bigg|_{z=0}\bigg)\hat{V}_1^{\gamma,k}(0)-\int_0^H\bigg(\frac{d\hat{V}_1^{\gamma,k}}{dz}\bigg)^2dz+\bigg[\frac{\mu_1'(H)}{\mu_1(H)}(\hat{V}_1^{\gamma,k}(H))^2-\frac{\mu_1'(0)}{\mu_1(0)}(\hat{V}_1^{\gamma,k}(0))^2\bigg]\\&-\frac{1}{2}\int_0^H\frac{d}{dz}\bigg(\frac{\mu_1'(z)}{\mu_1(z)}\bigg)\bigg(\hat{V}_1^{\gamma,k}\bigg)^2 dz+\int_0^H\bigg(\frac{\rho_1(z)\omega^2}{\mu_1(x)}-k^2\bigg) \bigg(\hat{V}_1^{\gamma,k}\bigg)^2 dz=\int_0^H\mathscr{R}_{ode}\hat{V}_1^{\gamma,k}dz\end{align*}
Applying Young's inequality and using \ref{ode3} and \ref{ode4}, we get 
\begin{align*}&\mathscr{R}_{b2}\bigg(\frac{d\hat{V}_1^{\gamma,k}}{dz}\bigg|_{z=H}\bigg)-\mathscr{R}_{b1}\hat{V}_1^{\gamma,k}(0)-\int_0^H\bigg(\frac{d\hat{V}_1^{\gamma,k}}{dz}\bigg)^2dz+\bigg[\frac{\mu_1'(H)}{\mu_1(H)}(\hat{V}_1^{\gamma,k}(H))^2-\frac{\mu_1'(0)}{\mu_1(0)}(\hat{V}_1^{\gamma,k}(0))^2\bigg]\\&-\frac{1}{2}\int_0^H\frac{d}{dz}\bigg(\frac{\mu_1'(z)}{\mu_1(z)}\bigg)\bigg(\hat{V}_1^{\gamma,k}\bigg)^2 dz+\int_0^H\bigg(\frac{\rho_1(z)\omega^2}{\mu_1(x)}-k^2\bigg) \bigg(\hat{V}_1^{\gamma,k}\bigg)^2 dz\leq\frac{1}{2}\int_0^H\mathscr{R}_{ode}^2dz+\frac{1}{2}\int_0^H\bigg(\hat{V}_1^{\gamma,k}\bigg)^2dz\end{align*}
\begin{align*}\implies \int_0^H\bigg[\bigg(\frac{\rho_1(z)\omega^2}{\mu_1(z)}-k^2\bigg)-\frac{d}{dz}\bigg[\frac{\mu_1'(z)}{\mu_1(z)}\bigg]-\frac{1}{2}\bigg]\bigg(\hat{V}_1^{\gamma,k}\bigg)^2 dz \leq & \mathscr{R}_{b1}\hat{V}_1^{\gamma,k}(0)-\mathscr{R}_{b2}\bigg(\frac{d\hat{V}_1^{\gamma,k}}{dz}\bigg|_{z=H}\bigg)\\&+\int_0^H \bigg(\hat{V}_1^{\gamma,k}\bigg)^2 dz+\frac{1}{2}\int_0^H(\mathscr{R}_{ode})^2 dz.\\&+ \bigg[\frac{\mu_1'(0)}{\mu_1(0)}(\hat{V}_1^{\gamma,k}(0))^2-\frac{\mu_1'(H)}{\mu_1(H)}(\mathscr{R}_{b2})^2\bigg]\end{align*}
\begin{align*}\implies \int_0^H\bigg(\hat{V}_1^{\gamma,k}\bigg)^2 dz \leq & \frac{1}{I(\omega,k)}\bigg\{\mathscr{R}_{b1}\hat{V}_1^{\gamma,k}(0)-\mathscr{R}_{b2}\bigg(\frac{d\hat{V}_1^{\gamma,k}}{dz}\bigg|_{z=H}\bigg)+\int_0^H \bigg(\hat{V}_1^{\gamma,k}\bigg)^2 dz\\&+ \bigg[\frac{\mu_1'(0)}{\mu_1(0)}(\hat{V}_1^{\gamma,k}(0))^2-\frac{\mu_1'(H)}{\mu_1(H)}(\mathscr{R}_{b2})^2\bigg]+\frac{1}{2}\int_0^H(\mathscr{R}_{ode})^2 dz\bigg\}.\end{align*}
Using the estimates in Lemma~\ref{lemma5}, \ref{quad1} and \ref{quad2} we get the desired estimate.
\end{proof}
\subsection {Computation of Dispersion Relation} \label{section_bisection}
Since the activation functions $\sigma$ are continuous, the trained neural network approximation $V_{1,\theta}^{\gamma,k}$ to the solution of our ODE is well-defined and continuous for all
\[
k\in
\bigg(
\omega\sqrt{\frac{\rho_2}{\mu_2}},
{\omega}\bigg/
\sqrt{\displaystyle\inf_{z\in[0,H]}
\frac{\mu_1(z)}{\rho_1(z)}}
\bigg].
\]
\begin{remark}
    It is worth noting that there is no need to exclude the exceptional set $\mathscr{A}$\eqref{disp_main}, since the neural network approximation $V_{1,\theta}^{\gamma,k}$ is well-defined and continuous for all $k$ in the prescribed interval.
\end{remark}
For a fixed $\omega$, the function
\begin{equation}
d_{\omega}(k)=\mu_1(H)\frac{dV_{1,\theta}^{\gamma,k}}{dz}(H)
+\mu_2\sqrt{k^2-\frac{\rho_2\omega^2}{\mu_2}}
\label{nndiap}\end{equation}
is continuous in the above domain. Therefore, we apply the bisection method to find its roots, thereby obtaining an approximate solution of the dispersion relation \eqref{disp_main}. 
\section{Numerical Experiments} \label{sec4}
In this section we present two applications of the proposed method by plotting their dispersion curves after obtaining the neural network trained solution $V_{1,\theta}^{\gamma,k}$ and using \eqref{nndiap}. Separate neural networks were trained for $\omega\in\{2.5,5,7.5,...,22.5\}\text{ rad/s}$ with step size 2.5 rad/s. The network details are given in Table~\ref{tab:pinn_summary}. The first 10,000 epochs were used as a warm-up phase, during which the learning rate was held fixed and the loss term weights were kept at their initial values. Several architectural and training configurations were explored during development and the setup described in Table~\ref{tab:pinn_summary} was found to give the most stable and accurate convergence for obtaining the dispersion curves according to Section~\ref{section_bisection}. For each frequency $\omega$, the trained network was saved as a separate \texttt{.pth} checkpoint file, which was later loaded to evaluate the network solution over a range of $k$ values and identify the admissible wavenumbers satisfying the dispersion relation. All models were implemented in PyTorch, with training carried out on Kaggle, where each frequency took approximately 8-11 hours to train. A schematic representation is given in Fig~\ref{nn}.

\begin{figure}
    \centering
    \includegraphics[width=1\linewidth]{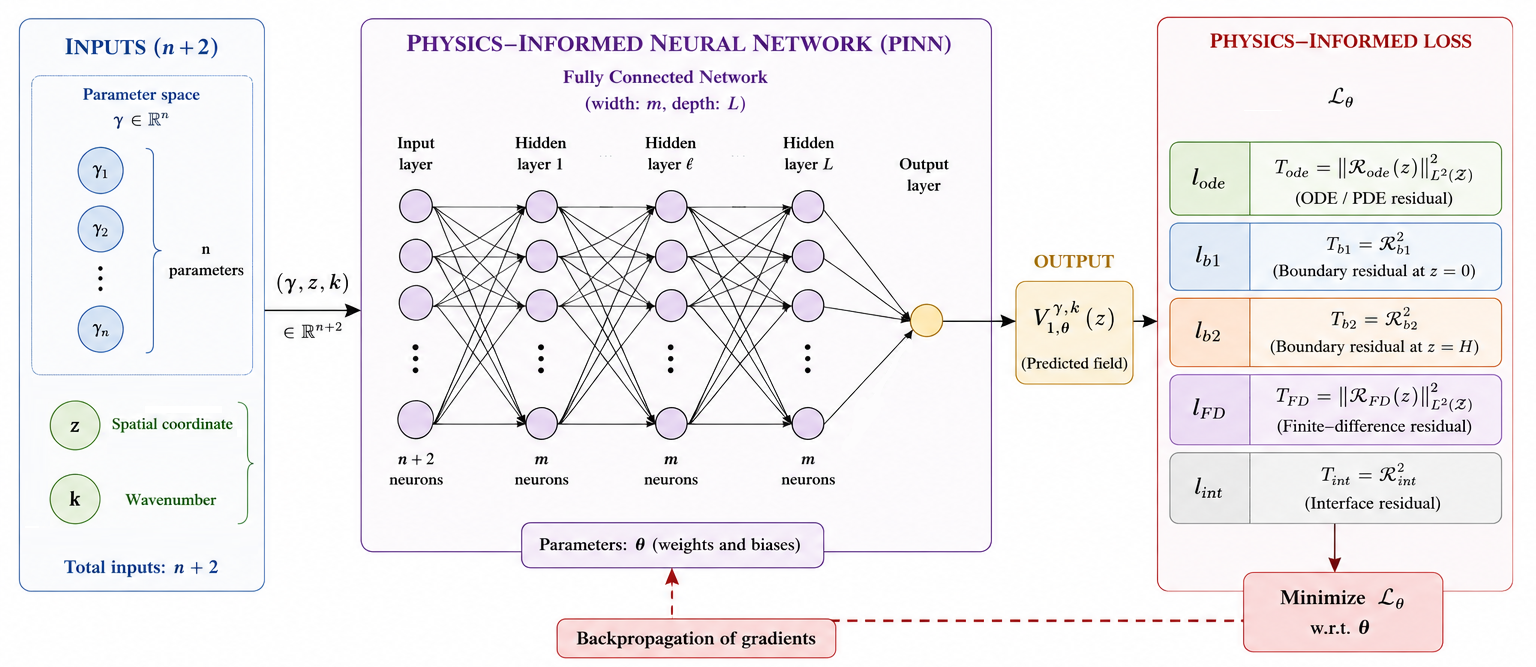}
    \caption{Schematic of the proposed physics-informed neural network.}
    \label{nn}
\end{figure}
\begin{table}[htbp]
\centering
\caption{Summary of PINN architecture and training configuration.}
\label{tab:pinn_summary}
\begin{tabular}{@{}p{4cm}p{10cm}@{}}
\toprule
\textbf{Component} & \textbf{Details} \\
\midrule
Network Architecture & Fully connected feedforward network with residual connections:input layer ($n+2 \to 100$), 4 hidden residual layers ($100 \to 100$ each) and output layer ($100 \to 1$). \\
\addlinespace
Activation Function & SiLU (Swish), chosen for smoothness and differentiability required for second-order derivatives in the $\mathscr{R}_{ode}$ residual. \\
\addlinespace
Loss Function &$\mathscr{L}_{\theta}$, where $T_{ode}$ and $T_{FD}$ is calculated by composite trapizoidal rule.\\
\addlinespace
Loss Weighting & Adaptive, inverse-magnitude rebalancing every 500 epochs. \\
\addlinespace
Training Points & 4000 equally spaced points for $z\in[0,H]$. \\
\addlinespace
Optimizer (Phase 1) & Adam, initial learning rate $5 \times 10^{-4}$, gradient clipping (max norm 1.0) \\
\addlinespace
LR Schedule & 10{,}000-epoch linear warmup, followed by cosine annealing to $1 \times 10^{-6}$. \\
\addlinespace
Epochs (Phase 1) & 100{,}000 \\
\addlinespace
Optimizer (Phase 2) & L-BFGS (strong Wolfe line search, history size 50, max 20 inner step) \\
\addlinespace
Steps (Phase 2) & 50 \\
\addlinespace
Batch Size & 10 parameter combinations per epoch \\
\bottomrule
\end{tabular}
\end{table}
\subsection{SH Waves in Layered Earth Media}
We first test the model in a seismological setting where a heterogeneous sandstone layer sits on top of an elastic granite halfspace. The shear modulus and density have exponential changes with rates $a$ and $b$ respectively. Specifically, $\mu_1(z)=\mu_0 e^{az}$ and $\rho_1(z)=\rho_0 e^{bz}$ $(\therefore n=2)$. The parameter values are shown in Table~\ref{tab:material_properties}. Since this setup does not have any analytical solution for the dispersion relation, we compare it to a known analytical case, when $a=b=0$ (classical Love Wave \eqref{Love}) and $a=b$ \eqref{case1}. For the general case, we plot results against the Heskell matrix method \cite{haskell1990dispersion}. We first break the layer into $10$ homogeneous sections, taking the average shear modulus and density for each layer. Then we do this for $50$ and $100$ layers. We compare these results with the continuous case predicted by our model. Finally, we conduct a parameter sensitivity analysis to see how changes in $a$ and $b$ affects the wave propagation. The training loss as a function of the number of epochs for $\omega=2.5$ is shown in Fig.~\ref{fig:loss}. The loss decreases significantly during training, with all error components reaching $<10^{-5}$ in the end. The other neural networks exhibited similar training behavior and comparable error levels; hence, only a representative training-loss plot is presented.
\begin{table}[htbp]
    \centering
    \caption{Material properties used in the numerical simulations for SH wave in layered earth media.}
    \label{tab:material_properties}
    \begin{tabular}{cccc}
        \hline
        $\rho_0 (\text{kg}/m^3 )$ \cite{mavko2009rock}& $\rho_1(\text{kg}/m^3 )$ \cite{yu2014thermophysical} & $\mu_0(\text{N}/m^2)$\cite{rice2021manufacture} & $\mu_1(N/m^2)$ \cite{yu2014thermophysical} \\
        \hline
        $2370$ & $2540$ & $0.139\times10^{10}$ & $1.32\times10^{10}$ \\
        \hline
    \end{tabular}
\end{table}
From Figs~\ref{fig:validation}(a) and \ref{fig:validation}(b), we can see that our PINN based solution agrees well with the analytical solution with Root Mean Square (RMS) errors of 0.1024 and 0.1312 respectively. As an additional validation, the heterogeneous layer is approximated by an increasing number of homogeneous sublayers and the corresponding dispersion curves are computed using the Haskell matrix method. With increasing number of sublayers, the Haskell matrix solution converges to the dispersion relation obtained using the proposed method as shown in Fig~\ref{fig:Haskell}. 

Further, the parameter sensitivity analysis shown in Fig.~\ref{fig:a_b_variation}(a) shows that the phase velocity increases with $a$. This is because $a$ controls the depth-dependent shear modulus, $\mu_1(z)=\mu_0e^{az}$. As $a$ increases, the layer becomes stiffer, allowing shear waves to propagate faster. In contrast, as shown in Fig.~\ref{fig:a_b_variation}(b), increasing $b$ increases the density, $\rho_1(z)=\rho_0e^{bz}$, which increases the inertia of the medium and hence reduces the wave speed. It can also be seen from both figures that the effect is more pronounced at higher frequencies. This is because $a$ and $b$ describe the properties of the finite layer rather than the half-space and higher-frequency waves are more sensitive to the shallow part of the medium. These results further confirm that the model captures the expected physical behavior.

\begin{figure}[htbp]
    \centering
    \begin{minipage}{0.45\textwidth}
        \centering
        \includegraphics[width=\textwidth]{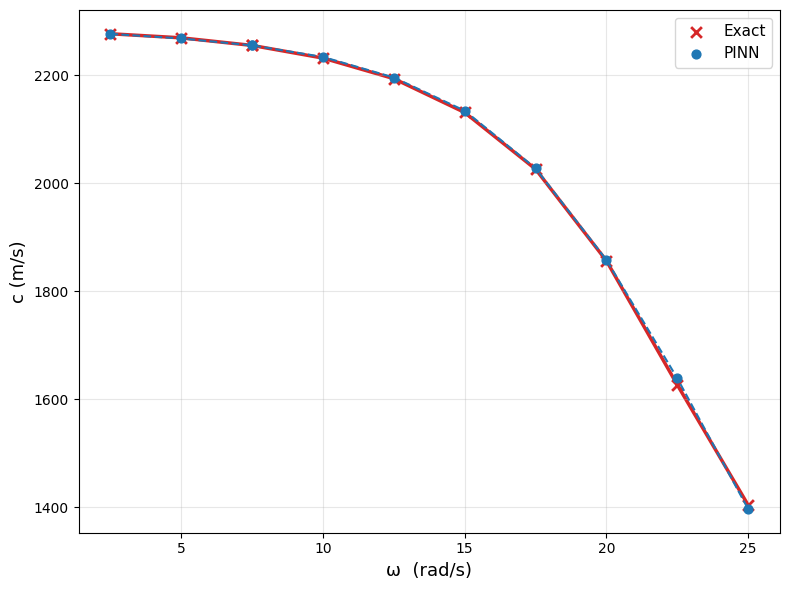}
        \smallskip
        \text{(a) Comparison with classical Love wave.}
    \end{minipage}
    \hfill
    \begin{minipage}{0.45\textwidth}
        \centering
        \includegraphics[width=\textwidth]{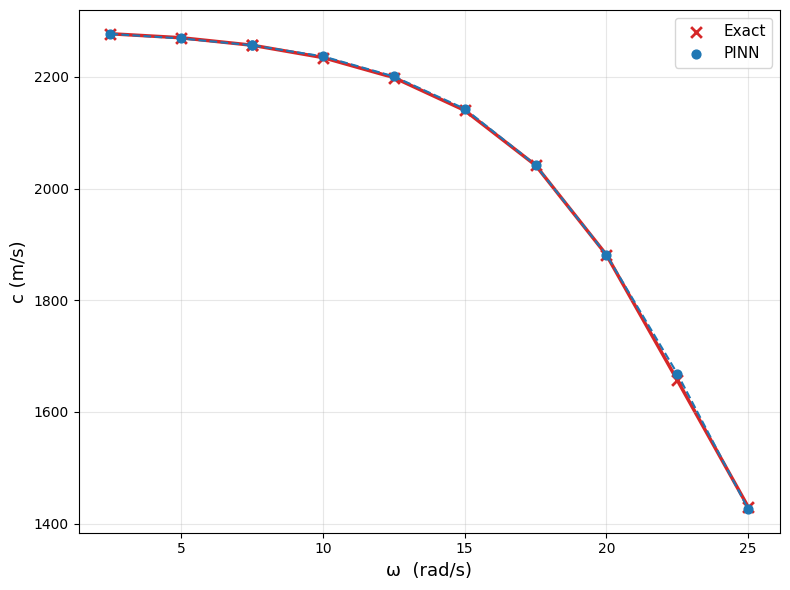}
        \smallskip
        \text{(b) Comparison with the case when $a=b=0.001$.}
    \end{minipage}

    \vspace{0.3cm}
   \caption{Validation of the PINN solution.}
    \label{fig:validation}
\end{figure}
\begin{figure}
    \centering
    \includegraphics[width=0.45\linewidth]{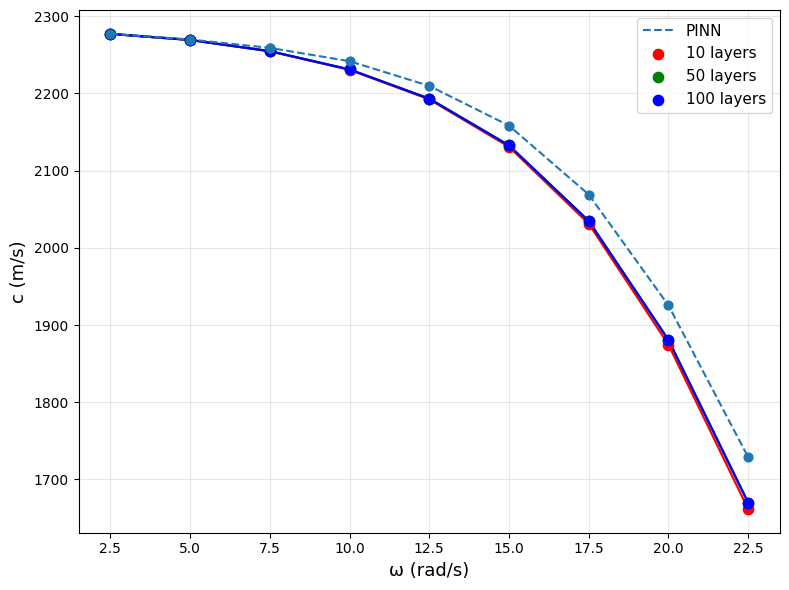}
    \caption{Comparison with the dispersion curves derived using the Haskell matrix method}
    \label{fig:Haskell}
\end{figure}
\begin{figure}[htbp]
    \centering
    \begin{minipage}{0.45\textwidth}
        \centering
        \includegraphics[width=\textwidth]{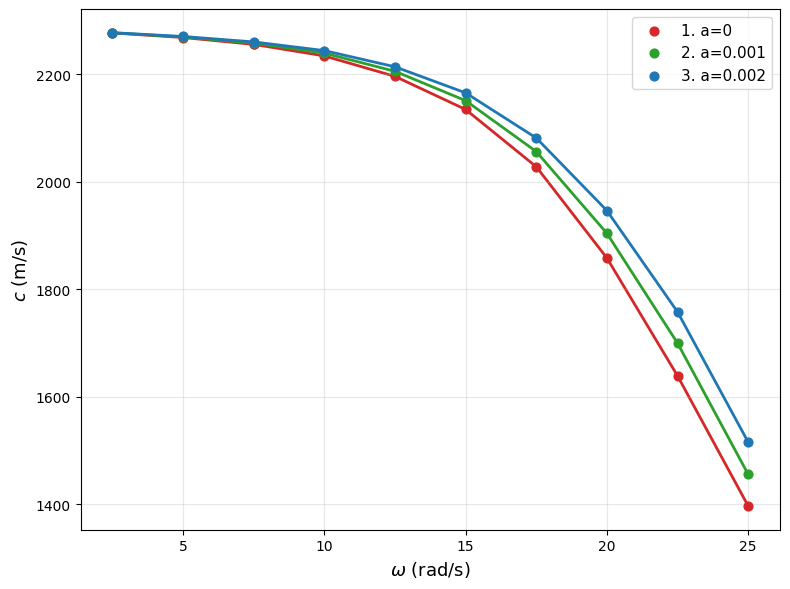}
        \smallskip
        \text{(a) Effect of change in $a$.}
    \end{minipage}
    \hfill
    \begin{minipage}{0.45\textwidth}
        \centering
        \includegraphics[width=\textwidth]{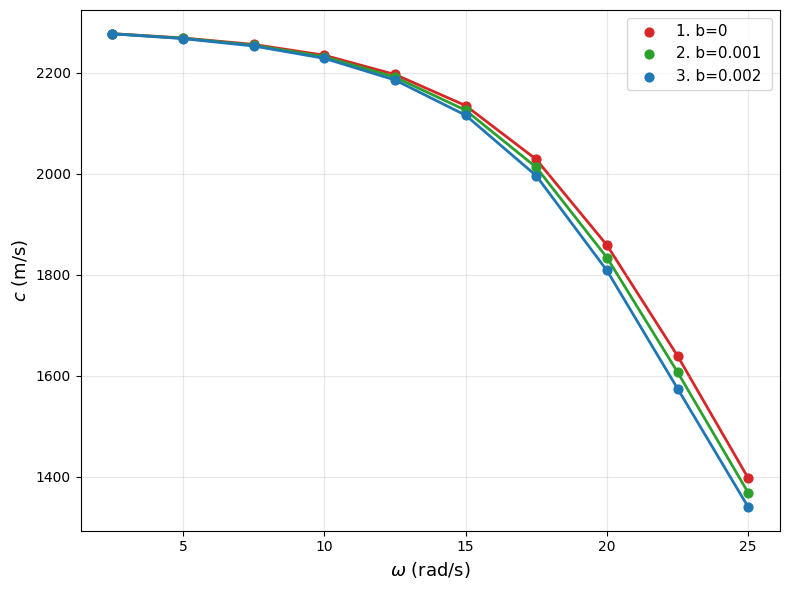}
        \smallskip
        \text{(b) Effect of change in $b$.}
    \end{minipage}
    \caption{Variation of parameter $a$ and $b$.}
    \label{fig:a_b_variation}
\end{figure}
\begin{figure}
    \centering
    \includegraphics[width=0.75\linewidth]{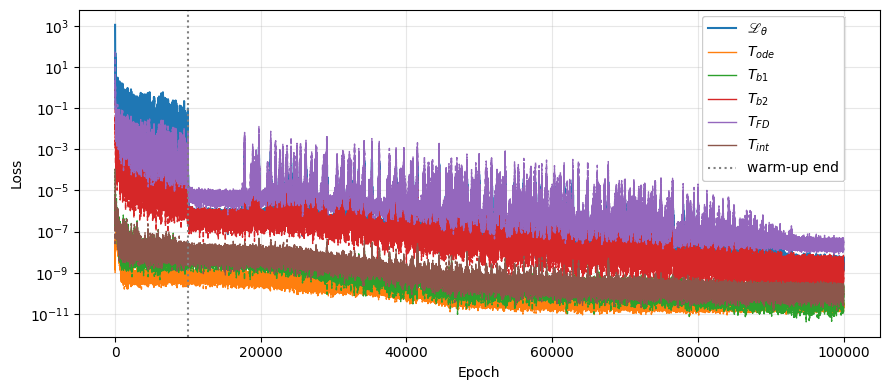}
    \caption{Number of epochs vs training losses.}
    \label{fig:loss}
\end{figure}

\subsection{SH Waves in Layered Media with Piezoelectric and Piezomagnetic Substrates}
The main aim of this section is to highlight an important advantage of this formulation. We consider a hetrogeneous elastic layer overlying either a piezomagnetic or piezoelectric substrate, which are governd by different physical laws. One thing that is common in both is the hetrogeneous layer equation and can be solved independently using the PINN structure and subsequently coupled together into the corresponding dispersion relation, as illustrated schematically in Fig.~\ref{fig:piezo_mag}.
\begin{table}[htbp]
    \centering
    \caption{Material properties used in the numerical simulations for the SH wave in piezoelectric and piezomagnetic layered configuration.}
    \label{tab:material_properties}
    \begin{tabular}{cccc}
        \hline
        \multicolumn{4}{c}{\text{Elastic layer \cite{pramanik2016propagation}}}\\
        \hline
        $\mu_0\;(\mathrm{N/m^2})$ &
        $\rho_0\;(\mathrm{kg/m^3})$ & & \\
        \hline
        $3.5\times10^9$ & $1600$ & & \\
        \hline
        \multicolumn{4}{c}{\text{Piezoelectric substrate: BaTiO$_3$ \cite{sadab2025wave}}}\\
        \hline
        $c_{44}\;(\mathrm{N/m^2})$ &
        $e_{15}\;(\mathrm{C/m^2})$ &
        $\varepsilon_{33}\;(\mathrm{C^2/(N\,m^2)})$ &
        $\rho_2^e\;(\mathrm{kg/m^3})$\\
        \hline
        $43\times10^9$ &
        $11.6$ &
        $11.2\times10^{-9}$ &
        $5800$\\
        \hline
        \multicolumn{4}{c}{\text{Piezomagnetic substrate: CoFe$_2$O$_4$ \cite{pang2008propagation}}}\\
        \hline
        $c_{44}\;(\mathrm{N/m^2})$ &
        $h_{15}\;(\mathrm{N/(A\,m)})$ &
        $\mu_{11}\;(\mathrm{N\,s^2/C^2})$ &
        $\rho_2^m\;(\mathrm{kg/m^3})$\\
        \hline
        $45.3\times10^9$ &
        $550$ &
        $157\times10^{-6}$ &
        $5300$\\
        \hline
    \end{tabular}
\end{table}

We consider the exponentially heterogeneous elastic layer of varying rates of hetrogeniety of shear modulus and density, ($\mu_1(z)=\mu_0e^{az}$ and $\rho_1(z)=\rho_0e^{bz}$), overlying either a piezoelectric or a piezomagnetic substrate. For the governing equations and underlying physical formulation, we refer the reader to (\cite{sadab2025dispersive,singh2026nonlocal}) and the references therein. We now present the dispersion relations electrically short (piezoelectric) and magnetically short (piezomagnetic) cases using our formulation. The approximate dispersion relation for the SH wave propagation on hetrogeneous layer overlying a piezoelectric substrate which is electrically shorted is thus given as follows:
\begin{equation}
\mu_1(H)\frac{\bm{d}\bm{V}_{\bm{1,\theta}}^{\bm{\gamma,k}}\bm{(H)}}{\bm{dz}}
+c_{44}\alpha_e
+\frac{e_{15}\beta_e(\alpha_e-k)}
{\alpha_e^2-k^2}
=0
\label{pes}\end{equation}
where $\alpha_e^2
=
k^2-
\frac{\rho_2\omega^2}
{c_{44}+\dfrac{e_{15}^2}{\varepsilon_{33}}}
$ and $ 
\beta_e
=
\frac{e_{15}}{\varepsilon_{33}}
\label{pms}\left(\alpha_e^2-k^2\right).
$
Here, $c_{44}$ denotes the shear elastic stiffness constant of the
piezoelectric medium, $e_{15}$ is the piezoelectric
coupling coefficient, $\varepsilon_{33}$ is the dielectric permittivity in the
$z$-direction and $\rho_2^e$ is the density of the piezoelectric substrate. Similarly replacing the piezoelectric with piezomagnetic substrate we get the following approximate dispersion relation, for the magnetically short case:
\begin{equation}
\mu_1(H)\frac{\bm{d}\bm{V}_{\bm{1,\theta}}^{\bm{\gamma,k}}\bm{(H)}}{\bm{dz}}
+c_{44}\alpha_m
+\frac{h_{15}k e^{-H(k-\alpha_m)}}{\gamma_m}
+\frac{h_{15}\alpha_m\beta_m}
{\alpha_m^2-k^2}
=0
\label{pms}
\end{equation}
where
$\alpha_m^2
=
k^2-
\frac{\rho_2\omega^2}
{c_{44}+\dfrac{h_{15}^2}{\mu_{11}}},
$ $\beta_m
=
\frac{h_{15}}{\mu_{11}}
\left(\alpha_m^2-k^2\right)
$
and
$\gamma_m
=
\frac{
k(\alpha_m^2-k^2)\mu_{11}
}{
e^{H(k-\alpha_m)}
\left\{
\alpha_m
\left[
h_{15}(\alpha_m^2-k^2)
-\beta_m\mu_{11}
\right]
\right\}
}.
$
Here, $c_{44}$ denotes the shear elastic stiffness constant of the piezomagnetic medium, $h_{15}$ is the piezomagnetic coupling coefficient, $\mu_{11}$ is the magnetic permeability in the $z$-direction and $\rho_2^m$ is the density of the piezomagnetic substrate. Note that in both the cases \eqref{pes} and \eqref{pms}, the same $\bm{V}_{\bm{1},\bm{\theta}}^{\bm{\gamma},\bm{k}}$ is being used which is the PINN solution of \ref{ode1}. for this numerical experiments  the data used is given in Table~. The training loss for the neural network  $\omega=2.5$ is shown in Fig~\ref{fig:loss2}. For validation, we compared the dispersion relations derived from \eqref{pms} and \eqref{pes} with their corresponding analytical cases $(a=b=0)$, as shown in Fig.~\ref{fig:validation2}(a) and (b). The corresponding RMS errors are 0.8945 and 0.1328, respectively.   

A natural extension of the proposed formulation is to consider electrically open and magnetically open boundary conditions for the piezoelectric and piezomagnetic substrates, respectively. The dispersion relation of the heterogeneous layer on piezoelectric substrate under electrically open boundary conditions can therefore be derived as follows:
\begin{equation}
\mu_1(H)\frac{\bm{d}\bm{V}_{\bm{1,\theta}}^{\bm{\lambda,k}}\bm{(H)}}{\bm{dz}}+c_{44}\alpha_e+\frac{e_{15}k e^{-H(k-\alpha_e)}}{\gamma_e}+\frac{e_{15}\alpha_e\beta_e}{\alpha_e^2-k^2}=0
\label{peo}
\end{equation}
where
 $\gamma_e
=
\frac{
k(\alpha_e^2-k^2)\varepsilon_{33}
}{
e^{H(k-\alpha_e)}
\left\{
\alpha_e
\left[
e_{15}(\alpha_e^2-k^2)
-\beta_e\varepsilon_{33}
\right]
\right\}
}.$

Similarly, the dispersion relation for the hetrogeneous layer over piezomagnetic substrate under magnetically open boundary conditions is given as follows:
\begin{equation}
\mu_1(H)\frac{\bm{d}\bm{V}_{\bm{1,\theta}}^{\bm{\lambda,k}}\bm{(H)}}{\bm{dz}}
+c_{44}\alpha_m
+\frac{h_{15}\beta_m(\alpha_m-k)}
{\alpha_m^2-k^2}
=0.
\label{pmO}
\end{equation}
 From Figs.~\ref{pes} (a), \ref{pms} (a), \ref{peo} (a) and \ref{pmO} (a), we observe a common trend: as the heterogeneity parameter $(a)$ associated with the rate of change of the shear modulus increases, the phase velocity also increases. Similarly, from Figs.~\ref{pes} (b), \ref{pms} (b), \ref{peo} (b) and \ref{pmO} (b), we observe that the phase velocity decreases as the heterogeneity parameter $(b)$ associated with the density increases. Moreover, significant changes in the phase velocity are observed in both cases, particularly in the high-frequency range. The physical reasoning behind these observations has already been discussed in the previous section.     
\begin{figure}[htbp]
    \centering
    \includegraphics[width=0.8\linewidth]{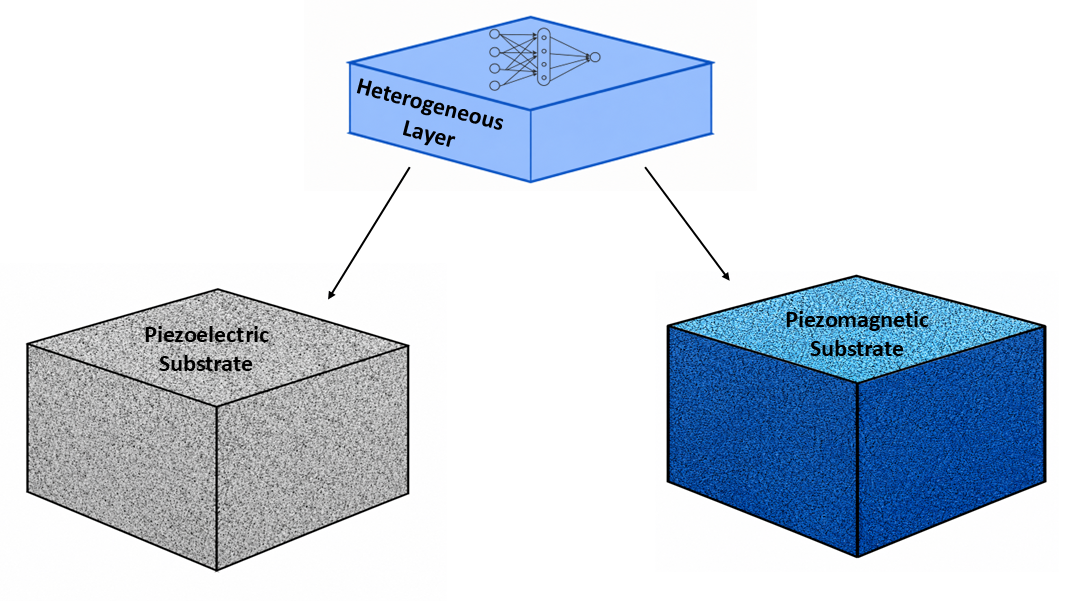}
    \caption{Schematic representation of the proposed formulation for piezoelectric and piezomagnetic substrates}
    \label{fig:piezo_mag}
\end{figure}

\begin{figure}[htbp]
    \centering
    \begin{minipage}{0.45\textwidth}
        \centering
        \includegraphics[width=\textwidth]{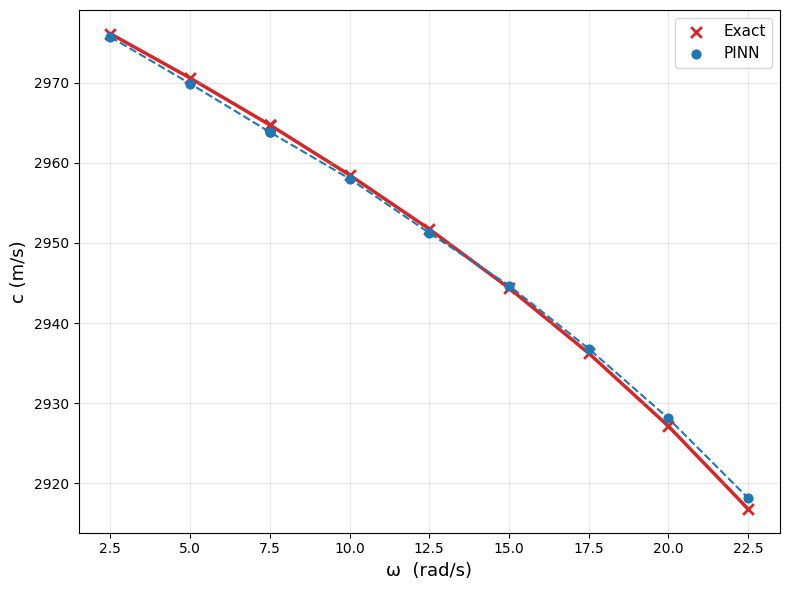}
        \smallskip
        \text{(a) Comparison with analytical }
        \text{piezoelectric case $(a=b=0)$.}
    \end{minipage}
    \hfill
    \begin{minipage}{0.45\textwidth}
        \centering
        \includegraphics[width=\textwidth]{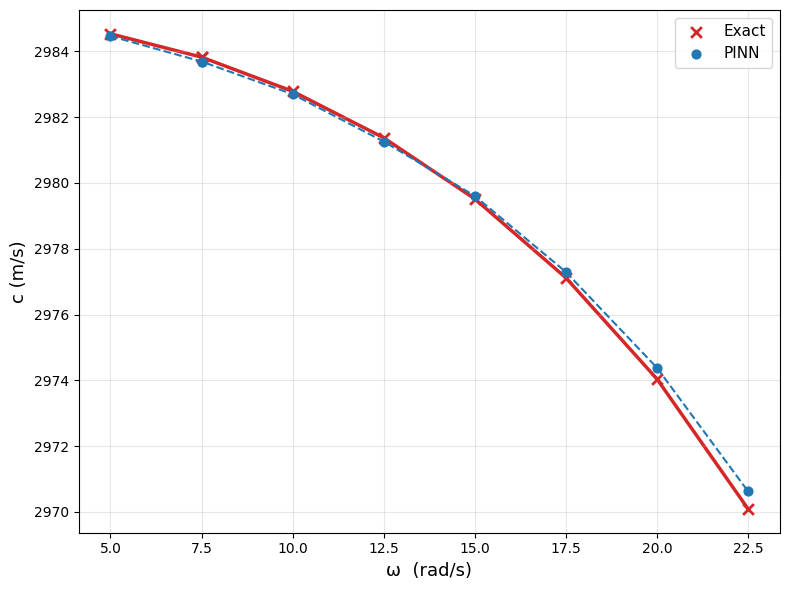}
        \smallskip
        \text{(b) Comparison with analytical} 
        \text{piezomagnetic case $(a=b=0)$.}
    \end{minipage}

    \vspace{0.3cm}
   \caption{Validation of the PINN solution.}
    \label{fig:validation2}
\end{figure}
\begin{figure}[htbp]
    \centering
    \begin{minipage}{0.45\textwidth}
        \centering
        \includegraphics[width=\textwidth]{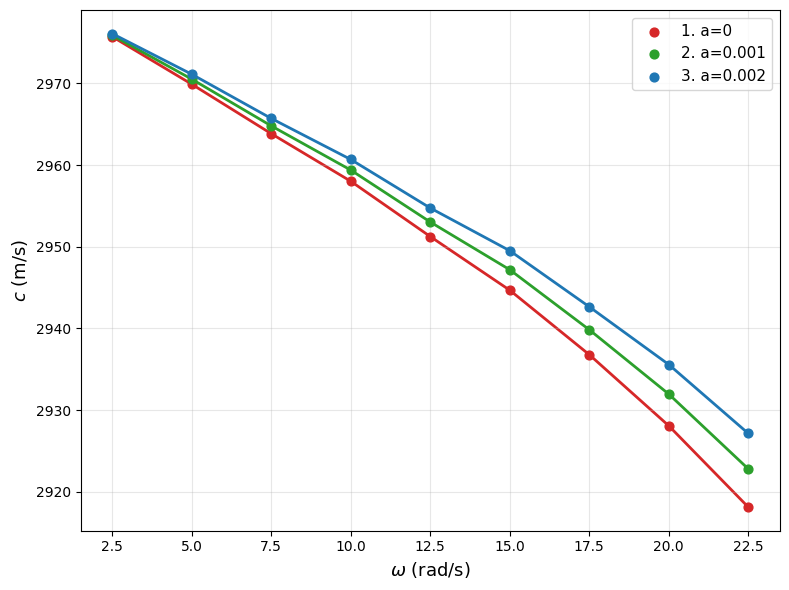}
        \smallskip
        \text{(a) Effect of change in a.}
    \end{minipage}
    \hfill
    \begin{minipage}{0.45\textwidth}
        \centering
        \includegraphics[width=\textwidth]{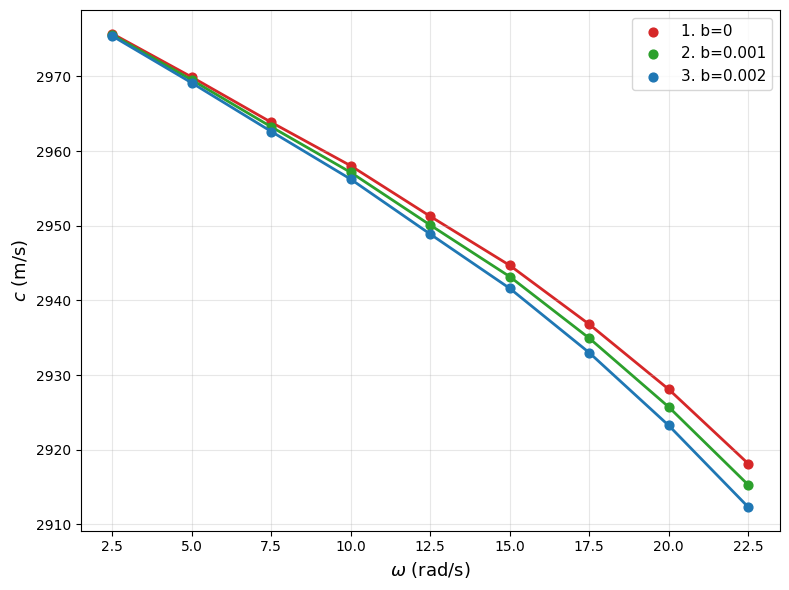}
        \smallskip
        \text{(b) Effect of change in b.}
    \end{minipage}

    \vspace{0.3cm}
   \caption{Variation of parameter a and b in piezoelectric short case.}
\end{figure}   
\begin{figure}[htbp]
    \centering
    \begin{minipage}{0.45\textwidth}
        \centering
        \includegraphics[width=\textwidth]{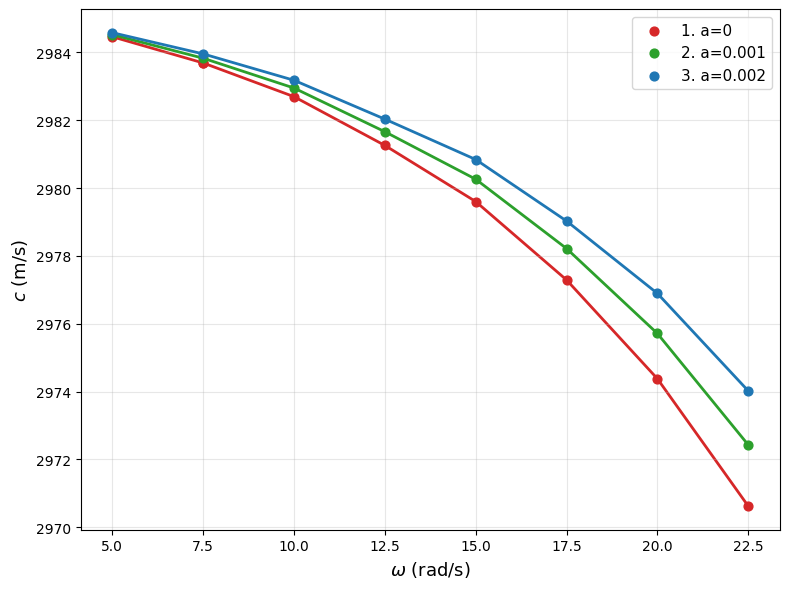}
        \smallskip
        \text{(a) Effect of change in a.}
    \end{minipage}
    \hfill
    \begin{minipage}{0.45\textwidth}
        \centering
        \includegraphics[width=\textwidth]{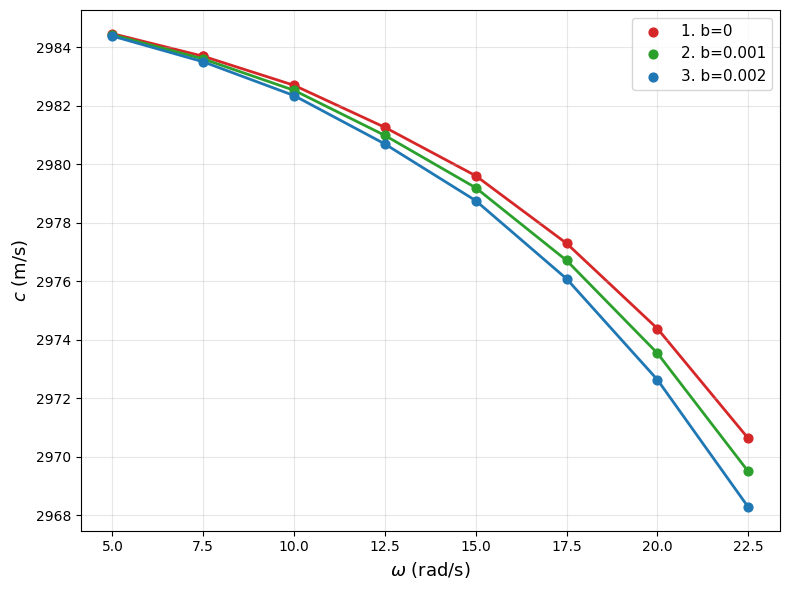}
        \smallskip
        \text{(b) Effect of change in b.}
    \end{minipage}

    \vspace{0.3cm}
   \caption{Variation of parameter a and b in piezomagnetic short case.}
\end{figure}

\begin{figure}[htbp]
    \centering
    \begin{minipage}{0.45\textwidth}
        \centering
        \includegraphics[width=\textwidth]{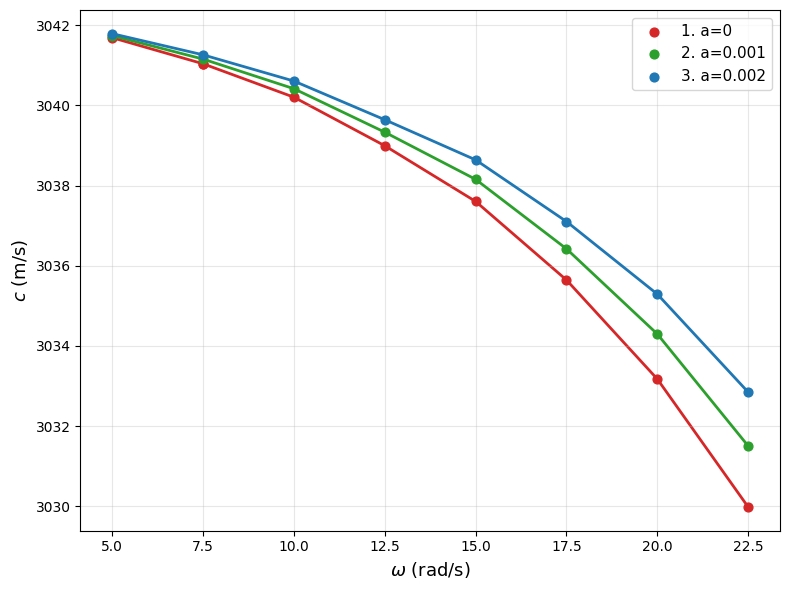}
        \smallskip
        \text{(a) Effect of change in a.}
    \end{minipage}
    \hfill
    \begin{minipage}{0.45\textwidth}
        \centering
        \includegraphics[width=\textwidth]{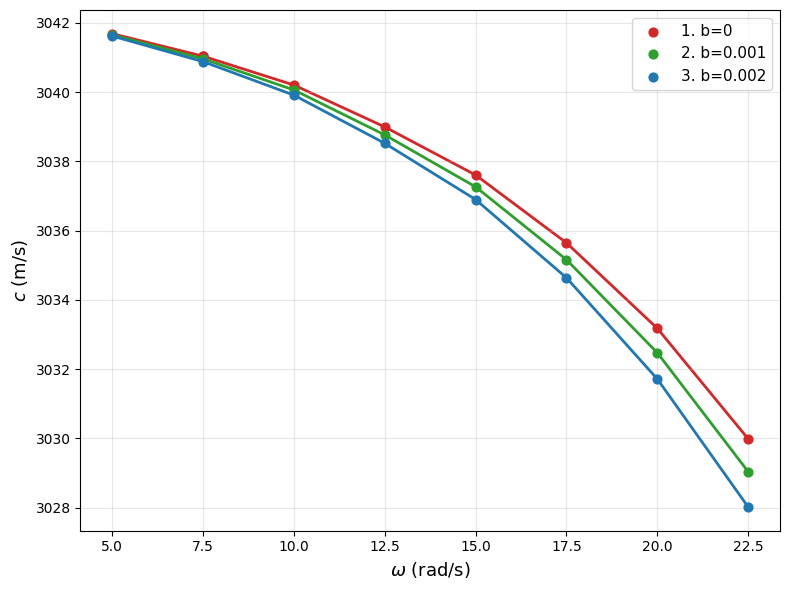}
        \smallskip
        \text{(b) Effect of change in b.}
    \end{minipage}

    \vspace{0.3cm}
   \caption{Variation of parameter a and b in piezoelectric open case.}
\end{figure}   
\begin{figure}[htbp]
    \centering
    \begin{minipage}{0.45\textwidth}
        \centering
        \includegraphics[width=\textwidth]{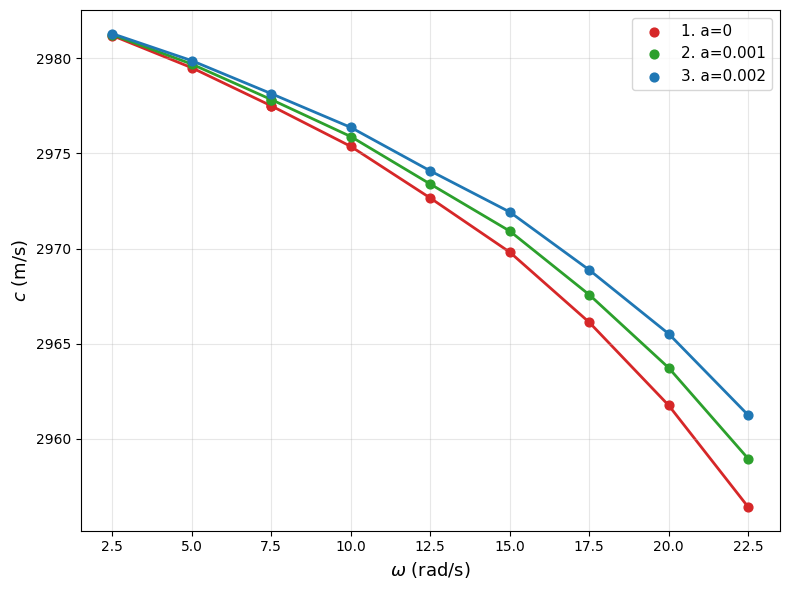}
        \smallskip
        \text{(a) Effect of change in a.}
    \end{minipage}
    \hfill
    \begin{minipage}{0.45\textwidth}
        \centering
        \includegraphics[width=\textwidth]{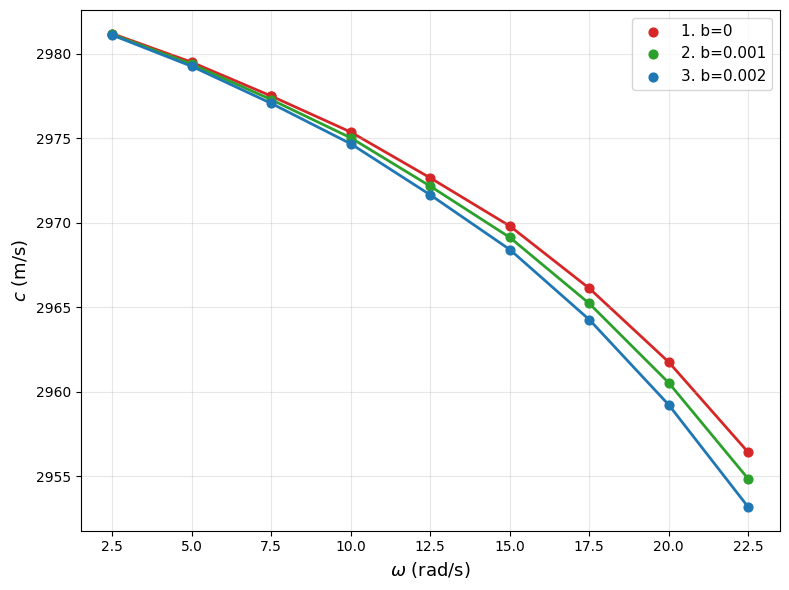}
        \smallskip
        \text{(b) Effect of change in b.}
    \end{minipage}

    \vspace{0.3cm}
   \caption{Variation of parameter a and b in piezomagnetic open case.}
\end{figure}  
\begin{figure}
    \centering
    \includegraphics[width=0.75\linewidth]{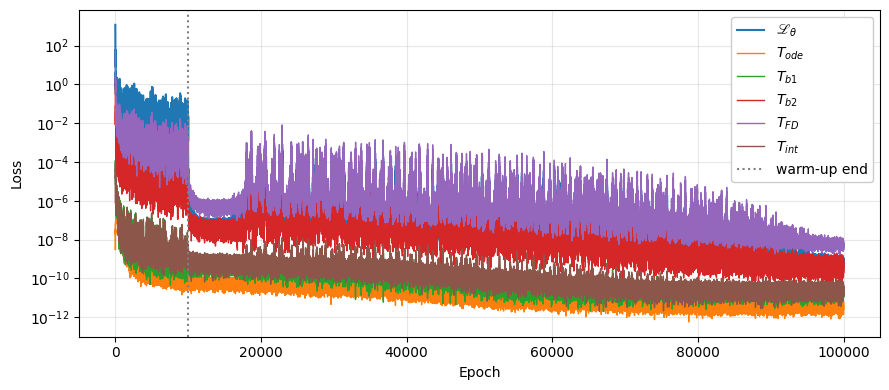}
    \caption{Number of epochs vs training losses.}
    \label{fig:loss2}
\end{figure}

\section{Conclusion} \label{sec5}

In this work, we propose a physics-informed neural network (PINN) framework for determining the dispersion relations of SH waves propagating in a continuously heterogeneous elastic layer over different underlying substrates.  The heterogeneous-layer problem was separated from the substrate contribution, allowing the same learned layer representation to be coupled with different substrate models, including elastic, piezoelectric and piezomagnetic substrates. The proposed formulation was validated against analytical solutions and the Haskell matrix method. The numerical results showed the expected influence of the heterogeneity parameters on the phase velocity.

The mathematical properties of the formulation were also investigated. We established that the singular set of the finite-difference system has measure zero, providing a basis for constructing stable numerical training data. Furthermore, a generalization-error estimate was derived for the PINN approximation, accounting for the neural-network approximation, quadrature, finite-difference discretization and interface residuals. These results provide a theoretical basis for assessing the accuracy and stability of the proposed framework.

A further advantage of the proposed formulation is its computational efficiency and reusability. While the conventional finite-difference solution requires $\mathcal{O}(N)$ operations with respect to the number of spatial grid points $N$, the trained PINN provides the layer response and its interface derivative with an $\mathcal{O}(1)$ online cost with respect to $N$. This makes the framework particularly useful for repeated dispersion calculations and for investigating the heterogeneity of a layer when its thickness is known. In particular, the trained network can serve as a reusable representation of the heterogeneous layer rather than requiring the boundary-value problem to be solved from scratch for every new substrate or parameter configuration.

A natural extension of the present framework is to compute the dispersion relation of torsional surface waves Dey et al.~\cite{dey1996torsional} with general heterogeneity. In the formulation of dispersion relation, the torsional displacement is represented as $v(r,z,t)=V(z)J_1(kr)e^{-i\omega t}$ and the resulting equation for the depth-dependent function $V(z)$ has the same mathematical structure as the heterogeneous-layer ODE considered in this work \eqref{ode1}, with $k$ representing the radial wavenumber rather than the horizontal wavenumber of the SH-wave problem. Thus, the present PINN formulation provides a natural basis for extending the approach to torsional waves and other related wave configurations. Future work will focus on developing such reusable PINN models and making the trained models and associated implementations available through a repository, enabling the learned heterogeneous-layer response to be repeatedly coupled with different substrates and wave configurations without solving the heterogeneous-layer problem from scratch.
\section*{Acknowledgments}
The authors gratefully acknowledge the Indian Institute of Technology (Indian School of Mines), Dhanbad, for providing the necessary research facilities.



\end{document}